\documentclass[a4paper,11pt]{article}
\usepackage{fullpage}
\usepackage{amsmath,amsfonts,amsthm,amssymb}
\usepackage{graphicx}
\usepackage{caption}
\usepackage{subcaption}

\numberwithin{equation}{section}
\newlength{\parindentsave}

\usepackage{esint} 

\providecommand{\R}{\mathbb{R}}
\providecommand{\A}{\mathcal{A}_L}

\providecommand{\E}{R}

\providecommand{\I}{\frac{\pi \mathbb{Z}}{L}} 
\providecommand{\Ip}{\frac{\pi \mathbb{N}}{L}} 
\providecommand{\ve}{\varphi_\epsilon}

\providecommand{\vt}{\Upsilon}

\renewcommand{\S}{\mathcal{S}_L}

\providecommand{\ud}{\, \mathrm{d}}
\providecommand{\dxy}{\ud x \ud y}
\providecommand{\dx}{\ud x}
\providecommand{\dy}{\ud y}

\providecommand{\Oomega}{\Omega'}

\newcommand{\supp}{\operatorname{supp}} 
\providecommand{\lint}{\fint_{-L}^L}

\newtheorem{theorem}{Theorem}
\newtheorem{lemma}{Lemma}[section]
\newtheorem*{lemma*}{Lemma}
\newtheorem{remark}[lemma]{Remark}
\newtheorem{proposition}[lemma]{Proposition}
\newtheorem{cor}[lemma]{Corollary}
\newtheorem*{def*}{Definition}

\title{Transition between planar and wrinkled regions in uniaxially stretched thin elastic film}
\author{Peter Bella\footnote{Max Planck Institute for Mathematics in the Sciences, Leipzig (Germany), email: bella@mis.mpg.de}}

\begin{document}
\allowdisplaybreaks 

\maketitle

\begin{abstract}
We study the transition from flat to wrinkled region in uniaxially stretched thin elastic film. We set up a model variational problem, and study energy of its ground state. Using known scaling bounds for the minimal energy, the minimal energy can be written as a minimum of the underlying (convex) relaxed problem plus a term, which grows linearly in the thickness of the film. We show that in the limit of vanishing thickness the prefactor in the scaling law for the original problem can be obtained by minimization of simpler {\it scalar} constrained variational problems. 
\end{abstract}


\section{Introduction}\label{intro}

In the last few years understanding the behavior of deformed thin elastic sheets has attracted a lot of attention both in the mathematics and physics communities (see, e.g., \cite{bib-surfacewrinkling} and references therein). Both physical and mathematical approach have a common starting point -- they consider a minimization of a suitable elastic energy, which in a simplified setting consists of a non-convex ``membrane energy'' plus a higher-order singular perturbation representing the ``bending energy.'' Since the energy is non-convex there can be many local minimizers or even non-minimizing stationary points, 
and so we should not use just the local optimality conditions (Euler-Lagrange equations). 
One possible approach to the problem is to guess specific form of a minimizer, and use this {\it ansatz} to find a deformation with minimal energy in this restricted (usually much smaller) class of deformations -- this approach is often favored by physicists (see, e.g., \cite{bib-bennynew,bib-benny1}). 

Different approach to the minimization of non-convex variational problem is to focus on (i) the minimum value of the energy, and (ii) the properties of low-energy states. This idea is often used to study energies which depend on some small or large physical parameter (or possibly several parameters), with the minimum of the energy estimated in terms of this parameter (often as a power-law). Such an approach was first introduced by Kohn and M\"uller in \cite{bib-bob+muller1+,bib-bob+muller1} to study a model problem for the fine scale structure of twinning near an austenite--twinned-martensite interface, where they identified a scaling law for the energy minimum in terms of material parameter. Since then this approach was successfully applied to many variational problems arising in material science \cite{choksi2001scaling,choksi2008ground,choksi1999domain}.

The present paper is motivated by the study of deformations of an annular elastic thin film (see~\cite{bib-bellakohn2,bib-bennynew,bib-benny1,bib-geminard}). Dead loads applied both on the inside and outside boundary cause the film to wrinkle in some region. Indeed, if the loads inside are large enough compared to the loads on the outer boundary, the deformation in the radial direction forces the concentric circles of the material near the inner boundary to decrease their length by more than is required by the Poisson ratio of the material. Therefore, the membrane needs to waste this excess in the circumference
either by compression or by buckling out of plane, contributing to the energy with
some amount that depends on $h$, the thickness of the film. In~\cite{bib-bellakohn2}, Kohn and the author identify the optimal scaling law for the energy. More precisely, it is shown that the minimum of the average energy per unit thickness satisfies $\mathcal{E}_0 + C_0 h \le \min E_h \le \mathcal{E}_0 + C_1 h$, where $\mathcal{E}_0$ is the energy required to stretch the film radially and the next term (linear in $h$) is the correction due to wrinkling. In contrast, the construction used in~\cite{bib-benny1} requires energy at least  $\mathcal{E}_0 + C h \log h$, which is suboptimal for small enough $h$.  

As mentioned above the sheet prefers to wrinkle near the inner boundary because in there  some arclength in the circumferential direction should be wasted. In fact, there are two regions -- the inner region where the arclength should be wasted and the outer region where the sheet is stretched biaxially. These two regions are basically separated by a circle (``free boundary''), and the amount of arclength to be wasted in the inner region growths linearly in the distance from this circle. The ansatz used in~\cite{bib-benny1} assumes that the number of wrinkles is constant, which together with the linear growth of the wasted arclength causes additional logarithm in the scaling law, whereas the construction used in \cite{bib-bellakohn2} increases the number of wrinkles near the transition from wrinkled to planar region.\footnote{As a matter of fact, if the amount of arclength would grow slower (e.g. quadratically in the distance), both constructions would achieve the same scaling law, and there would be no 
reason to vary the number of wrinkles.} 
Though the cascade of wrinkles used in~\cite{bib-bellakohn2} achieves the optimal scaling for the energy, it is not clear whether such a construction is necessary. 

To understand better what happens in the transition region we set up a toy problem, which should be simple enough to allow an analysis beyond the scaling law, but which should still possess all the important features of the original problem. We consider an elastic thin film which is stretched in one direction, and in the other direction it wants to be stretched in half of the domain and needs to waste some arclength in the rest of the domain. To correctly model the setting, we assume that the amount by which the sheet is stretched (or the amount of arclength it needs to waste) is a linear function of the distance to the transition region. 

If $\min E_h$ denotes the minimum of the elastic energy for the sheet of thickness $h$, using the same approach as in~\cite{bib-bellakohn2} one can show that $\frac{1}{h} ( \min E_h -\mathcal{E}_0)$ is bounded from above and away from $0$ independently of the thickness. In the present paper we go a step further -- we show that the limit $\lim_{h \to 0} \frac{1}{h} ( \min E_h -\mathcal{E}_0)$ exists and that its value can be characterized by minimization of simpler (scalar) constrained variational problems. In other words, we prove that minimal values of the rescaled energies $(E_h - \mathcal{E}_0)/h$ have the same limit as minimal values of simpler scalar constrained functionals. It is then natural to ask whether a similar statement holds for minimizers, i.e., whether (possibly properly rescaled) minimizers of $E_h$ converge (in some sense) to minimizers of scalar constrained functionals. We would like to address this question (using methods of Gamma-convergence) in a future work. 


To put our work in its proper context, let us mention related (both mathematics and physics) articles. In the mathematical community, there has been done a lot of work on scaling laws for problems arising in elasticity of thin films. We mention study of compressed thin films on a substrate (see \cite{bib-kohn-nguyen} and \cite{bib-bedrossiankohn} for two different settings) and of compressed thin film blisters \cite{bib-blisters-linear,bib-blisters,bib-jin-sternberg,bib-jin-sternberg2}, study of a cascade of wrinkles at the edge of floating elastic sheet \cite{bib-kohnnguyen-raft} (its motivation comes from~\cite{bib-benny-raft}) and work on elastic thin sheets with prescribed non-euclidean metric \cite{bib-bellakohn1}.
In all these problems, as 
$h \to 0$, deformations develop a microstructure. In particular, it is expected that  minimizers have complicated structure, due to which it is very difficult to study (or at least guess) the exact form of a minimizer -- hence in these problems it is hard to obtain more than just the energy scaling law. 

In a recent work, Brandman, Kohn, and Nguyen \cite{bib-brandman-kohn-nguyen} studied conical singularities in thin elastic sheets (so called ``d-cone'' problem). Since in this setting minimizers (as $h \to 0$) do not develop a microstructure (but rather a point singularity), it is possible to go little bit further and obtain the optimal prefactor in the energy scaling law (Olbermann and M\"uller \cite{bib-olbermannmuller-dcone} even managed to study the next term in the asymptotic expansion in $h$ of the ground state energy). 


Finally, let us mention two more mathematical studies of ground states in problems where microstructure is expected. In \cite{bib-contibranching}, Conti studied asymptotic behavior of a ground state for the ``Kohn-M\"uller'' energy (see \cite{bib-bob+muller1}). He used clever local constructions to obtain local energy bounds to show that the ground state is asymptotically self-similar. In a different setting, Otto and Viehmann~\cite{bib-ottoviehmann} analyzed the ground state energy of a ferromagnetic bulk sample with strong uniaxial anisotropy in a regime featuring domain branching. At first glance, this work looks very similar to ours -- they also anisotropically rescale the energy by the expected microstructure length scale and study the rescaled problem. But compare to our setting, they obtain simplified variational problems defined on a fixed domain, whereas due to a non-local constraint we need to consider energies defined on a sequence of increasing 
domains. 

Compared to our problem, in physics literature there are works with very similar physical setting, but are interested in different questions and use different methods -- as already mentioned, starting from an ansatz, they look for explicit solutions and study their asymptotic properties. In~\cite{bib-benny-fissioning}, Davidovitch studied transition between wrinkled profiles with different periods at two opposite sides of a rectangular sheet, which is stretched in the transversal direction. He was mainly interested in the dependence of the length scale of wrinkles on the position and the boundary conditions. A similar question was addressed in \cite{bib:bldgblock}, where the authors studied behavior of one ``building block'' -- a deformation, which has wrinkled profile with period $1$ and $3$ at two opposite sides of a rectangular sheet. To finish, we mention \cite{bib-benny1,bib-bennynew,bib-geminard}, where stretching of an annular elastic sheet 
was studied. 



The rest of this Introduction gives a brief overview of our main achievements and organization of the paper. We start Section~\ref{setting} with the proper definition of the elastic energy $E_h$. After that, we state our main result -- Theorem~\ref{thm1}, which basically says 
\begin{equation}\label{intro1}
 \lim_{h \to 0} \frac{\min E_h - \mathcal{E}_0}{h} = \sigma \in (0,\infty),
\end{equation}
where $\sigma = \inf_{L>0} \min_{u} \S(u)$ and $\{\S(u)\}_{L>0}$ is a family of scalar variational functionals, which are minimized over some restricted set of functions $\A$ (see~\eqref{AL}). To study minimization of $E_h$ we first perform an anisotropic rescaling by a factor $L := h^{-1/2}$ in the $y$-variable (see Section~\ref{heuristics} for the motivation for the rescaling and the actual rescaling). 

To prove~\eqref{intro1}, in Section~\ref{sect:ub} we show the upper bound $$\limsup_{h \to 0} h^{-1} (\min E_h - \mathcal{E}_0) \le \sigma.$$ 
More precisely, for any $\epsilon > 0$ and any $h \in (0,h(\epsilon))$, we construct a test deformation for $E_h$ with energy smaller than $\mathcal{E}_0 + (\sigma+\epsilon)h$. Our construction is based on $u$, a minimizer of $\S(u)$. In particular, for the construction we require some regularity estimates for $u$, which are the content of Theorem~\ref{thm:u}. 

In Section~\ref{sect:lb} we prove the matching lower bound 
$$\liminf_{h \to 0} h^{-1} (\min E_h - \mathcal{E}_0) \ge \sigma.$$
Given $h > 0$ and any deformation $(w,u_3)$ (which satisfies the boundary condition), we use $u_3$ to construct a test function $u \in \A$ for $\S$ such that $h^{-1} (E_h(w,u_3) - \mathcal{E}_0) \ge \S(u)$ plus an error, which goes to $0$ as $h \to 0$. Since $u$ should belong to $\A$, we show that either the energy $E_h(w,u_3)$ is too large, or $u_3$ almost belongs to $\A$. In the latter case, we slightly modify $u_3$ to obtain $u \in \A$ (see Lemma~\ref{lm:deltal}) while not increasing its energy too much. 

Section~\ref{sect:scalar} is devoted to the proof of Theorem~\ref{thm:u}. Fixing $L \ge 1$, we study minimization of $\S$. We show existence of a minimizer, and study its regularity. Using Fourier series in the $y$-variable we rewrite the functional $\S$ and its minimizer $u$, and express $u$ in terms of its Fourier coefficients $a_k=a_k(x)$. We then study properties of $a_k$, in particular we derive an ODE for $a_k$. Observing an important role for quantities $\mu_k$ defined through $a_k' = \mu_k a_k k^2$, we show 
exponential decay of frequencies with short wavelengths (large $k$), which in turn implies the desired regularity estimates. Finally, in the last section we present the proof of Lemma~\ref{lm:deltal}.

\vskip.2cm
\noindent
{\bf Notation:} Unless explicitly stated, all the functions $W_1,W_2,U_3$ and $w_1,w_2,u,u_3,v$ will be periodic in the $y$-variable with period equal to the size of their domain of definition. By $C,C_1,\ldots$ we will denote universal constants (always independent of $L$), which can possibly change their value from line to line. Also, by $a \lesssim b$, $a \gtrsim b$, $a \sim b$, we will mean $a \le C b$, $a \ge C b$, $b/C \le a \le C b$, respectively.


\section{Setting and the main result}\label{setting}

In this section we define the energy $E_h$ together with the assumptions we put on proper deformations. Before we state our main results, we shortly discuss the form of the energy and the motivation for the anisotropic rescaling. 

Let us now describe the precise setting. We consider a rectangular sheet $(-a,a) \times (-b,b)$ with prescribed non-euclidean metric of the form $m(x,y) = \dx^2 + (1+\kappa x) \dy^2$ (i.e., the reference configuration is not stress-free). Moreover, we assume the sheet is stretched in the $x$-direction (by applying dead loads at $x=\pm a$) and the deformation is periodic in the $y$-direction. Finally, we use a small-slope geometrically linear (F\"oppl-von K\'arm\'an) approximation for the elastic energy of the film. The average energy per unit thickness we consider has the form:
\begin{multline}\nonumber
\int_{-a}^a \int_{-b}^b \left|e(W) + \frac{1}{2} \nabla U_3 \otimes \nabla U_3 - \kappa x e_2 \otimes e_2 \right|^2 + h^2 \left| \nabla^2 U_3 \right|^2 \dxy - T \int_{-b}^b W_1(1,y) - W_1(-1,y) \dy, 
\end{multline}
where $W = (W_1(x,y),W_2(x,y)) : [-a,a] \times [-b,b) \to \R^2$ is the in-plane displacement with $e(W) = (\nabla W + \nabla^t W)/2$ being its symmetric gradient, and $U_3 = U_3(x,y) : [-a,a] \times [-b,b) \to \R$ is the out-of-plane displacement. We assume that the thickness of the film $h > 0$, $\kappa > 0$ describes the amount of shrinking in the metric, and we applied dead loads of strength $T > 0$. We assume $W$ and $U_3$ satisfy the periodic boundary conditions, i.e., $W$ and $U_3$ are $2b$-periodic functions in the $y$-variable. 

Using rescaling $(W(x,y),U_3(x,y)) = (\kappa^{-1} \hat W(\kappa x,\kappa y),\kappa^{-1} \hat U_3(\kappa x, \kappa y)$ we see that without loss of generality we can assume $\kappa = 1$. Similarly, using rescaling $(W,U_3) = (T \hat W,T^{1/2} \hat U_3)$ we can assume that $T = 2$. Finally, to simplify the notation we will assume that $a=b=1$, i.e. we consider only the domain $[-1,1] \times [-1,1)$ (as we will see, this restriction will not play any role in the argument). Assuming this, for $h>0$ we define 
\begin{multline}\label{pbm1}
 E_{h}(W,U_3) := \\ \frac{1}{2} \int_{-1}^{1} \int_{-1}^{1} \left|e(W) + \frac{1}{2} \nabla U_3 \otimes \nabla U_3 - x e_2 \otimes e_2 \right|^2 + h^2 \left| \nabla^2 U_3 \right|^2 \dxy - \int_{-1}^1 W_1(1,y) - W_1(-1,y) \dy,
\end{multline}
where 
$(W,U_3)$ are as above ($2$-periodic functions in the $y$-variable), and the energy is normalized per unit length in the $y$-variable (hence the prefactor $1/2$). 


Based on similar arguments as in \cite{bib-bellakohn2} it is expected that there exist a number $\mathcal{E}_0$ and two constants $0 < C_0 < C_1 < \infty$ such that for any $h > 0$:
\begin{equation}\label{pbm2}
 \mathcal{E}_0 + C_0 h \le \min_{(W,U_3)} E_h(W,U_3) \le \mathcal{E}_0 + C_1 h.
\end{equation}
In our setting, the number $\mathcal{E}_0 = -5/3$ is the minimum of the relaxed energy $$\frac{1}{2} \left[ \int_{\Oomega} \left(e(W) - \kappa x e_2 \otimes e_2 \right)_+^2 \dxy - 2\int_{-1}^1 W_1(1,y) - W_1(-1,y) \dy \right],$$
where $(A)_+ := \min_{M \ge 0} |M+A|$ denotes the non-negative part of the matrix. 
To prove the right-hand side inequality one constructs a deformation with energy less than $\mathcal{E}_0 + C_1h$. It is easy to see that this deformation should have small (ideally vanishing) out-of-plane displacement $U_3$ in the region $x < 0$. Moreover, a simple scaling argument suggests that wrinkles near $x = 1$ should have length scale of order $h^{1/2}$. To study the limit as $h \to 0$, we prefer this length scale to be of order $1$. To achieve this we rescale anisotropically in the $y$-variable by a factor $h^{-1/2}$. Using this method we obtain our main result:
\begin{theorem}\label{thm1}
 There exists $0 < \sigma < \infty$ such that
\begin{equation}\nonumber
 \lim_{h \to 0} \frac{\min_{(W,U_3)} E_h(W,U_3) - \mathcal{E}_0 }{h} = \sigma,
\end{equation}
where we minimize over functions $W=(W_1,W_2),U_3 : [-1,1] \times [-1,1) \to \mathbb{R}$, which are $2$-periodic in the $y$-variable, and $\mathcal{E}_0 := -5/3$. Moreover,
\begin{equation}\nonumber
 \sigma = \inf_{L>0} \sigma_L, 
\end{equation}
where 
\begin{equation}\label{sigmal}
 \S(u):= \int_0^1 \fint_{-L}^L u_{,x}^2 + u_{,yy}^2 \dxy,\quad \sigma_L := \inf_{u \in \A} \S(u)
\end{equation}
with 
\begin{equation}\label{AL}
\A := \left\{ \begin{matrix} u : [0,1]\times[-L,L) \to \R, u(0,\cdot)=0,\\ \mathrm{for \ (a.e.)\ }x \in [0,1] : u(x,\cdot) \mathrm{\ is\ } 2L\mathrm{-periodic\ and\ } \fint_{-L}^L u_{,y}^2(x,y) \dy = 2x 
\end{matrix} \right\}.
\end{equation}
\end{theorem}

In the proof of Theorem~\ref{thm1} we will need some properties of a global minimizer (ground state) of the energy defined in~\eqref{sigmal}, which should hold independently of $L \ge 1$:

\begin{theorem}\label{thm:u}
 Let $L \ge 1$. Then there exists a global minimizer $u \in \A$ of the energy 
 \begin{equation}\nonumber
  \S(u) = \int_0^1 \lint u_{,x}^2 + u_{,yy}^2 \dxy,
 \end{equation}
 $u$ is odd in the $y$-variable, and for every $x \in (0,1)$ satisfies:
 \begin{align}
  \lint u^2(x,y) \dy &\le C x^2 (|\ln x|+1),\label{reg1}\\
  \lint u_{,x}^2(x,y) \dy &\le C (|\ln x|+1),\label{reg2}\\  
  \lint u_{,xx}^2(x,y) \dy &\le C x^{-2} (|\ln x|^{7}+1),\label{reg3}\\
  \lint u_{,xy}^2(x,y) \dy &\le C x^{-1} (|\ln x|^2+1),\label{reg4}\\
  \lint u_{,yy}^2(x,y) \dy &\le C (|\ln x|+1),\label{reg5}\\
  \lint u_{,xyy}^2(x,y) \dy &\le C x^{-2} (|\ln x|^{3}+1),\label{reg6}
 \end{align} 
 where $C$ does not depend on $L$. 
\end{theorem}


\section{Rescaling and some heuristics}\label{heuristics}


In this section we will perform the anisotropic rescaling in the $y$-variable and compute the rescaled energy $\E(w,u_3)$. After that we present some heuristic arguments to identify which terms in the rescaled energy $\E$ are important (i.e., the ones which in the end contribute to $\S$) and which are not important (i.e., those which vanish in the limit $h \to 0$). We conclude this section with some simple observations (Lemma~\ref{lm:sigmal}) about $\sigma_L$ for different $L$, which we will need later. 

Let $h > 0$ be fixed. Given $(W,U_3)$, periodic functions in $y$, we define a rescaled deformation $(w,u_3)$, defined in the rescaled domain
\begin{equation}\nonumber
 \Omega := [-1,1] \times [-L,L], \quad L := h^{-1/2}
\end{equation}
by
\begin{align*}
 w_1(x,y) &:= W_1(x,L^{-1}y), \\
 w_2(x,y) &:= L W_2(x,L^{-1}y), \\
 u_3(x,y) &:= L U_3(x,L^{-1}y).
\end{align*}
We express the energy $E_h(W,U_3)$ using the rescaled deformation $w_1,w_2,u_3$:
\begin{align}\nonumber
E_h(W,U_3) &= \int_{-1}^1 \fint_{-1}^1 \left|W_{1,x} + U_{3,x}^2/2\right|^2 + \frac{1}{2}\left|W_{1,y} + W_{2,x} + U_{3,x} U_{3,y}\right|^2 + \left|W_{2,y} + U_{3,y}^2/2 - x\right|^2 \dxy
\\ & \quad + h^2 \int_{-1}^1 \fint_{-1}^1 U_{3,xx}^2 + 2U_{3,xy}^2 + U_{3,yy}^2 \dxy - 2 \fint_{-1}^{1} W_1(1,y) - W_1(-1,y)  \dy \nonumber
\\ &= \!\int_{-1}^{1}\! \lint \left|w_{1,x} + \!L^{-2} u_{3,x}^2/2\right|^2\! +\! \frac{L^{-2}}{2}\left|L^2 w_{1,y} + w_{2,x} + u_{3,x} u_{3,y}\right|^2 + \left|w_{2,y} + u_{3,y}^2/2 - x\right|^2 \dxy \nonumber
\\ &\quad + L^{-4} \int_{-1}^{1} \lint \left( L^{-2} u_{3,xx}^2 + 2 u_{3,xy}^2 + L^{2} u_{3,yy}^2\right) \dxy - 2\lint w_1(1,y) - w_1(-1,y) \dy  \nonumber
\\ &= \int_{-1}^1 \lint \left( w_{1,x} + L^{-2} u_{3,x}^2/2 - 1 \right)^2 - 1 \dxy + \int_{-1}^0 \lint |w_{2,y} + u_{3,y}^2/2 - x|^2 \dxy \label{deriv15}
\\ &\quad+ \int_{0}^1 \lint |w_{2,y} + u_{3,y}^2/2 - x|^2 \dxy \label{deriv16}
\\ &\quad+ L^{-2} \int_{-1}^1 \lint \left|L^2 w_{1,y} + w_{2,x} + u_{3,x} u_{3,y}\right|^2/2 \dxy \label{deriv17}
\\ &\quad + L^{-2} \int_{-1}^1 \lint u_{3,x}^2 + u_{3,yy}^2 \dxy \label{deriv18}
\\ &\quad + L^{-4} \int_{-1}^1 \lint 2u_{3,xy}^2 + L^{-2} u_{3,xx}^2 \dxy \label{deriv19}
\\ &=: \E_{L}(w,u_3),\label{RLenergy}
\end{align}
where the term $\left( w_{1,x} + L^{-2} u_{3,x}^2/2 - 1 \right)^2$ in~\eqref{deriv15} is obtained by writing $w_1(1,y) - w_1(-1,y) = \int_{-1}^1 w_{1,x} \dx$ and by completing the square (after collecting different terms).

We see that all terms in~(\ref{deriv15}-\ref{deriv19}) (up to a constant term in \eqref{deriv15}) are non-negative. Hence, to minimize the energy it is important to understand (or at least guess) which terms can be made small (of order less than $h = L^{-2}$), and which terms will in the end contribute to the limiting energy. 
Heuristically we expect that 
(at least for large $L$):
\begin{itemize}
 \item the term $\left( w_{1,x} + L^{-2} u_{3,x}^2/2 - 1 \right)^2$ from~\eqref{deriv15} is small and $w_{2,y} + u_{3,y}^2/2 \approx 0$ in $[-1,0]\times[-L,L]$, therefore terms from~\eqref{deriv15} should behave like $\mathcal{E}_0$:
 \begin{multline}\nonumber
  \int_{-1}^1 \lint \left( w_{1,x} + L^{-2} u_{3,x}^2/2 - 1 \right)^2 - 1 \dxy + \int_{-1}^0 \lint |w_{2,y} + u_{3,y}^2/2 - x|^2 \dxy \\ \approx \int_{-1}^1 \lint -1 + \int_{-1}^0 \lint x^2 \dx = \mathcal{E}_0;
 \end{multline}
 \item $w_{2,y} + u_{3,y}^2/2 \approx x$ in $[0,1] \times [-L,L]$ and $-L^2 w_{1,y} \approx  w_{2,x} + u_{3,x} u_{3,y}$ in $[-1,1]\times[-L,L]$, and so both~\eqref{deriv16} are~\eqref{deriv17} are small. Because of the first requirement and periodicity of $w_2$ in $y$ one has $\lint u_{3,y}^2/2 \dy \approx x - \lint w_{2,y} \dy = x$;
 \item term~\eqref{deriv19} includes a small prefactor $L^{-4}$ (small compared to $L^{-2}$), and will not need to be considered.
\end{itemize}
Altogether we expect that if $w_1,w_2,u_3$ has (almost) optimal energy $\E_L$, then 
\begin{equation}\nonumber
 \E_L(w,u_3) = \mathcal{E}_0 + L^{-2} \int_0^1 \lint u_{3,x}^2 + u_{3,yy}^2 \dxy + o(L^{-2}),
\end{equation}
and $u_3$ has to satisfy $\lint u_{3,y}^2(x,y) \dy \approx 2x$. We observe that this is consistent with (and, in fact, precise version of this statement is equivalent to) Theorem~\ref{thm1}. 

Before we start with the proof of Theorem~\ref{thm1}, let us show the following simple properties of $\sigma_L$:
\begin{lemma}\label{lm:sigmal}
 Let $L > 0, \alpha \ge 1$, and $N \in \mathbb{N}$. Then 
\begin{equation}\label{sigma:growth}
 \sigma_{\alpha L} \le \alpha^2 \sigma_L
\end{equation}
and
\begin{equation}\label{sigma:dec}
 \sigma_{NL} \le \sigma_L.
\end{equation}
Moreover, 
\begin{equation}\label{sigmale}
 \sup_{L \ge 1} \sigma_L \le 4\sigma_1,
\end{equation}
and 
\begin{equation}\label{def:sigma}
 \sigma = \inf_{L \ge 1} \sigma_L = \lim_{L \to \infty} \sigma_L.
\end{equation}
\end{lemma}

\begin{proof}
 Given $\epsilon > 0$, let $u \in \A$ be such that  $\S(u) \le \sigma_L + \epsilon$. We define $v(x,y) := \alpha u(x,\alpha^{-1}y)$ and observe that $v \in \mathcal{A}_{\alpha L}$. Then
 \begin{equation}\nonumber
  \sigma_{\alpha L} \le \mathcal{S}_{\alpha L}(v) = \int_0^1 \fint_{-\alpha L}^{\alpha L} v_{,x}^2 + v_{,yy}^2 \dxy = \int_0^1 \lint \alpha^2 u_{,x}^2 + \alpha^{-2} u_{,yy}^2 \dxy \le \alpha^{2} \S(u) \le \alpha^{2} (\sigma_L + \epsilon).
 \end{equation}
 Since $\epsilon>0$ was arbitrary,~\eqref{sigma:growth} follows. To prove~\eqref{sigma:dec}, it is enough to observe that for $u \in \A$, one can use periodicity of $u$ in the $y$-variable to define periodic extension $\bar u$ of $u$ in $[0,1] \times [-NL,NL]$ such that  $\S(u) = \mathcal{S}_{NL}(\bar u)$. 

 Relation~\eqref{sigmale} immediately follows from~\eqref{sigma:growth} and~\eqref{sigma:dec}, and so it remains to prove~\eqref{def:sigma}. Using~\eqref{sigma:dec} we see that $\inf_{L > 0} \sigma_L = \inf_{L \ge 1} \sigma_L =  \liminf_{L \ge 1} \sigma_L$, in particular the first equality in~\eqref{def:sigma} follows. Hence, to prove the second equality in~\eqref{def:sigma} it is enough to show that $\lim_{L \to \infty} \sigma_L = \liminf_{L \to \infty} \sigma_L$. 

 Given $\epsilon > 0$ and a (large) integer $N_0$, let $L_0 \ge 1$ be such that $\sigma_{L_0} \le \sigma + \epsilon$. For any $L \ge N_0L_0$ we define $N := \lfloor L/L_0 \rfloor$ and $\alpha := L/(NL_0)$. Since $NL_0 \le L < (N+1)L_0$, we have $1 \le \alpha < 1+1/N \le 1+N_0^{-1}$. Then 
 \begin{equation}\nonumber
  \sigma_L = \sigma_{\alpha NL_0} \overset{\eqref{sigma:growth}}{\le} \alpha^2 \sigma_{NL_0} \overset{\eqref{sigma:dec}}{\le} \alpha^2 \sigma_{L_0} \le (1+N_0^{-1})^2(\sigma+\epsilon).
 \end{equation}
 In particular, we see that for all $L \ge N_0L_0$ we have $\sigma_L \le (1+N_0^{-1})^2(\sigma+\epsilon)$, which concludes the proof since $\epsilon > 0$ and integer $N_0$ can be chosen arbitrarily.
\end{proof}



\section{Upper bound}\label{sect:ub}

To prove the upper bound
\begin{equation}\nonumber
 \limsup_{h\to 0}  \frac{\min_{(W,U_3)} E_h(W,U_3) - \mathcal{E}_0}{h} \le \sigma
\end{equation}
we use a global minimizer (ground state) $u$ of $\S$ (existence of which is granted by Theorem~\ref{thm:u}) to define deformation $(w_1,w_2,u_3)$. To estimate the energy $\E_L(w,u_3)$ (see~\eqref{RLenergy}), we will need regularity estimates of a ground state $u$ of $\S$, which are the content of Theorem~\ref{thm:u}. 

Let $\epsilon > 0$ be fixed. To prove the upper bound it is enough to show that there exists $h_\epsilon > 0$ such that for every $0 < h \le h_\epsilon$ we have
\begin{equation}\nonumber
 \frac{1}{h} \left( \min_{(W,U_3)} E_h(W,U_3) - \mathcal{E}_0\right) \le \sigma + \epsilon.
\end{equation}
Since by~\eqref{def:sigma} $\lim_{L\to \infty} \sigma_L = \sigma$, we fix $L_\epsilon \ge 1$ such that for every $L \ge L_\epsilon$ we have $\sigma_L \le \sigma + \epsilon/2$. 
We define  $h_\epsilon := (N_\epsilon L_\epsilon)^{-2}$, where $N_\epsilon$ is a large integer to be chosen later. Then for $h \in (0,h_\epsilon)$ we can find $L_0 \in [L_\epsilon,2L_\epsilon]$ and an integer $N \ge N_\epsilon$ such that 
\begin{equation}\nonumber
 h^{-1/2} =: L = NL_0.
\end{equation}
 Further let $u \in \mathcal{A}_{L_0}$ be a ground state of $\mathcal{S}_{L_0}$. Then
\begin{equation}\label{sigmae}
 \int_0^1 \fint_{-L_0}^{L_0} u_{,x}^2 + u_{,yy}^2 \dx \dy = \sigma_{L_0} \le \sigma + \epsilon/2
\end{equation}
and $u$ satisfies~(\ref{reg1}-\ref{reg6}). 

For the rest of Section~\ref{sect:ub}, $h$, $N$, $L_0$, and $u$ will stay fixed. In the next part (Section~\ref{ss:defwu}) we use $u$ to define $w_1,w_2,u_3$. We then show periodicity in $y$ of $w_1,w_2,u_3$, so that we can use them as (proper) test deformation for $\E_L$. In Section~\ref{ss:estwu}, we conclude the proof of the upper bound by estimating all the terms in $\E_L(w,u_3)$. 


\subsection{Definition and periodicity of $(w,u_3)$}\label{ss:defwu}

Let $\varphi : (-\infty,\infty) \to [0,1]$, $\varphi|_{(-\infty,1/2]} = 0, \varphi|_{[1,\infty)} = 1$ be a smooth cutoff function. Then for given $0 < \delta < 1$ we define a rescaled cutoff function
\begin{equation}\nonumber
 \varphi_\delta(x) := \varphi(x/\delta), \quad x \in [-1,1].
\end{equation}
Using $\varphi_\delta$ we define for $(x,y) \in [-1,1]\times[-L,L]$ 
\begin{equation}
\begin{aligned}\label{defwu}
  u_3(x,y) &:= \varphi_\delta(x) u(x,y), \\
  w_2(x,y) &:= \varphi_\delta^2(x)xy - \int_0^y u_{3,y}^2(x,y')/2 \ud y' = \varphi_\delta^2(x) \left( xy - \int_0^y u_{,y}^2(x,y')/2 \ud y'\right),\\
  w_1(x,y) &:= x - L^{-2} \int_0^y w_{2,x}(x,y') + u_{3,x}(x,y')u_{3,y}(x,y') \ud y',
\end{aligned}
\end{equation}
where $u$ is a ground state of $\mathcal{S}_{L_0}$ (which is odd in the $y$-variable), extended periodically to $[0,1] \times [-L,L]$ (originally it was defined only in $[0,1]\times [-L_0,L_0]$), and by $0$ to $[-1,0]\times[-L,L]$. 

We observe that for $-1 \le x \le \delta/2$ we have $u_3(x,y) = 0$, hence $w_2(x,y)=0$ and $w_1(x,y)=x$. Before we estimate all the terms in the energy $\E_L(w,u_3)$ we want to verify that $u_3,w_1,w_2$ are $2L$-periodic functions in the $y$-variable. In fact, we will show that they are $y$-periodic with period $2L_0$. 

Indeed, since $u$ is $2L_0$-periodic in $y$, so is $u_3$. To show that $w_1$ and $w_2$ are periodic we use the following simple fact: a differentiable function $f$ is $\lambda$-periodic if $f'$ is $\lambda$-periodic and $f(t) = f(t+\lambda)$ for some $t$. We compute
\begin{equation}\nonumber
 w_2(x,L_0) - w_2(x,-L_0) = \varphi_\delta^2(x) \left( 2L_0x - \frac{1}{2}\int_{-L_0}^{L_0} u_{,y}^2(x,y) \dy  \right) = 0,
\end{equation}
where we used that $\fint_{-L_0}^{L_0} u_{,y}^2(x,y) \dy = 2x$. Moreover, $w_{2,y}(x,y) = \varphi_\delta^2(x)(x -  u_{,y}^2(x,y)/2)$ is periodic in $y$, which implies periodicity of $w_2$. It remains to show periodicity of $w_1$. 
We compute
\begin{equation}\nonumber
 w_1(x,L_0) - w_1(x,-L_0) = -L^{-2} \int_{-L_0}^{L_0} w_{2,x} + u_{3,x}u_{3,y} \dy.
\end{equation}
Since $u$ is an odd function of $y$, so is $u_3$, and the second integrand $u_{3,x}u_{3,y}$ is an odd function of $y$ as well. Moreover, $u_{3,y}^2$ is even function of $y$, hence the definition of $w_2$ implies that $w_2$ is also an odd function of $y$. Since both integrands in the above relation are odd functions of $y$ and, since we integrate over a symmetric interval around $0$, the integral vanishes. Finally, we compute $w_{1,y} = L^{-2}(w_{2,x} + u_{3,x}u_{3,y})$ and use that both $w_{2,x}$ and $u_{3,x}u_{3,y}$ are periodic in $y$ (the first is a consequence of the periodicity of $w_2$ in $y$), which concludes the proof of periodicity of $w_1$. 

\subsection{Estimating $\E_L(w,u_3)$}\label{ss:estwu}

We will now estimate all the terms in 
\begin{align}
 \E_{L}(w,u_3) &= \int_{-1}^1 \lint \left( w_{1,x} + L^{-2} u_{3,x}^2/2 - 1 \right)^2 \dxy - 2 + \int_{-1}^0 \lint |w_{2,y} + u_{3,y}^2/2 - x|^2 \dxy \label{deriv15a}
\\ &\quad+ \int_{0}^1 \lint |w_{2,y} + u_{3,y}^2/2 - x|^2 \dxy \label{deriv16a}
\\ &\quad+ L^{-2} \int_{-1}^1 \lint \left|L^2 w_{1,y} + w_{2,x} + u_{3,x} u_{3,y}\right|^2/2 \dxy \label{deriv17a}
\\ &\quad + L^{-2} \int_{-1}^1 \lint u_{3,x}^2 + u_{3,yy}^2 \dxy \label{deriv18a}
\\ &\quad + L^{-4} \int_{-1}^1 \lint 2u_{3,xy}^2 + L^{-2} u_{3,xx}^2 \dxy. \label{deriv19a}
\end{align}

\begin{itemize}
\item
We start with the first integral on~\eqref{deriv15a}. Using $(a+b)^2 \le 2a^2 + 2b^2$ we get that 
\begin{multline}\label{30}
 \int_{-1}^1 \lint \left( w_{1,x} + L^{-2}u_{3,x}^2/2 - 1\right)^2 \dxy \\ \le 2 \int_{-1}^1 \lint \left( w_{1,x} - 1\right)^2 \dxy + 2 \int_{-1}^1 \lint  L^{-4}u_{3,x}^4/4  \dxy.
\end{multline}
The estimate for the right-hand side follows from the two following lemmas (Lemma~\ref{lm31} and Lemma~\ref{lm32}):

\begin{lemma}\label{lm31}
 The first integral on the right-hand side of~\eqref{30} satisfies
 \begin{equation}
   \int_{-1}^1 \lint \left( w_{1,x} - 1 \right)^2 \dxy \lesssim \frac{L_0^4}{L^4} \delta^{-1} (|\ln \delta|^8 + 1).\label{24.1a}
 \end{equation}
\end{lemma}

\begin{proof}
Definition of $w_2$ (see \eqref{defwu}) and integration by parts imply
\begin{multline}\label{eq:mizne}
 w_{2,x}(x,y) = \left(\varphi_\delta^2(x)x\right)_{,x} y - \int_0^y u_{3,xy}(x,y') u_{3,y}(x,y') \ud y' \\ = \left(\varphi_\delta^2(x)x\right)_{,x} y + \left(\int_0^y u_{3,x}(x,y') u_{3,yy}(x,y') \ud y'\right) - u_{3,x}(x,y)u_{3,y}(x,y) + u_{3,x}(x,0)u_{3,y}(x,0).
\end{multline}
Since $u$ is an odd function of $y$ (in particular, $u_3(x,0) = u_{3,x}(x,0) = 0$ for any $x \in [-1,1]$), the last term in~\eqref{eq:mizne} vanishes, thus
\begin{equation}\nonumber
 w_{2,x}(x,y) + u_{3,x}(x,y)u_{3,y}(x,y) = \left(\varphi_\delta^2(x)x\right)_{,x} y + \left(\int_0^y u_{3,x}(x,y') u_{3,yy}(x,y') \ud y'\right).
\end{equation}
 Using this relation and the definition of $w_1$ we see that 
\begin{equation}\nonumber
 w_1(x,y) = x - L^{-2} \int_0^y \left[ \left(\varphi_\delta^2(x)x\right)_{,x} y' + \left(\int_0^{y'} u_{3,x}(x,y'') u_{3,yy}(x,y'') \ud y''\right) \right] \ud y',
\end{equation}
and
\begin{multline}\nonumber
 w_{1,x}(x,y) -1 = 
\\ - L^{-2} \int_0^y \left[ \left(\varphi_\delta^2(x)x\right)_{,xx} y' + \left(\int_0^{y'} u_{3,xx}(x,y'') u_{3,yy}(x,y'') + u_{3,x}(x,y'') u_{3,xyy}(x,y'')\ud y''\right) \right] \ud y'.
\end{multline}


We are now in position to estimate $(w_{1,x}(x,y) - 1)^2$. We observe that since $w_1$ (and in particular $w_{1,x}$) are $2L_0$-periodic in the $y$-variable, it is enough to estimate $(w_{1,x}(x,y) - 1)^2$ for $|y| \le L_0$. For $x \in [\delta,1]$ and $|y'| \le L_0$ we have 
\begin{align*}\nonumber
 &\left| \int_0^{y'} u_{3,xx}(x,y'') u_{3,yy}(x,y'') + u_{3,x}(x,y'') u_{3,xyy}(x,y'') \ud y'' \right| 
 \\ &\qquad \quad = \left| \int_0^{y'} u_{,xx}(x,y'') u_{,yy}(x,y'') + u_{,x}(x,y'') u_{,xyy}(x,y'') \ud y'' \right| 
 \\ & \qquad \underset{|y'|\le L_0}{\overset{\text{H\"older's}}{\le}} \left( \int_0^{L_0} u_{,xx}^2 \right)^{1/2} \left( \int_0^{L_0} u_{,yy}^2 \right)^{1/2} + \left( \int_0^{L_0} u_{,x}^2 \right)^{1/2} \left( \int_0^{L_0} u_{,xyy}^2 \right)^{1/2}
 \\ &\ \overset{\eqref{reg3},\eqref{reg5},\eqref{reg2},\eqref{reg6}}{\lesssim} L_0 x^{-1} (|\ln x|^4 + 1).
\end{align*}
If $x \in [\delta/2,\delta]$ and $|y'| \le L_0$, using chain rule to write $u_{3,x}$, $u_{3,xx}$, and $u_{3,xyy}$ in terms of derivatives of $\varphi_\delta$ and $u$, we obtain a similar estimate:
\begin{align*}\nonumber
 &\left| \int_0^{y'} u_{3,xx}(x,y'') u_{3,yy}(x,y'') + u_{3,x}(x,y'') u_{3,xyy}(x,y'') \ud y'' \right| 
  \\ &\le \left| \int_0^{y'} (\varphi_{\delta,xx}u + 2\varphi_{\delta,x}u_{,x} + \varphi_\delta u_{,xx}) \varphi_{\delta} u_{,yy} + (\varphi_{\delta,x}u + \varphi_\delta u_{,x})(\varphi_{\delta,x}u_{,yy} + \varphi_\delta u_{,xyy}) \ud y'' \right| 
  \\ &\lesssim L_0 x^{-1} (|\ln x|^4 + 1),
\end{align*}
where the last inequality follows as above from H\"older's inequality, \eqref{reg1}, \eqref{reg2}, \eqref{reg3}, \eqref{reg5}, and \eqref{reg6}. Hence for $x \in [\delta,1], y \in [-L_0,L_0]$ we get:
\begin{equation}\nonumber
 (w_{1,x}(x,y) - 1)^2 \lesssim L^{-4} \left| \int_0^y L_0 x^{-1}(|\ln x|^4 + 1) \dy \right|^2 \lesssim \left(\frac{L_0}{L}\right)^4 x^{-2} (|\ln x|^8 + 1).
\end{equation}
Similarly, for $x \in [\delta/2,\delta], y \in [-L_0,L_0]$ we get:
\begin{equation}\nonumber
 (w_{1,x}(x,y) - 1)^2 \lesssim L^{-4} \left| \int_0^y L_0\delta^{-1} + L_0 x^{-1}(|\ln x|^4 + 1) \dy \right|^2 \lesssim \left( \frac{L_0}{L}\right)^4 x^{-2}(|\ln x|^8 + 1),
\end{equation}
and~\eqref{24.1a} follows by adding and integrating two previous inequalities and using that $w_{1,x}(x,y) = 1$ for any $x \in [-1,\delta/2)$.
\end{proof}

\begin{lemma}\label{lm32}
 The second integral on the right-hand side of~\eqref{30} satisfies
 \begin{equation}
 \int_{-1}^1 \lint L^{-4} u_{3,x}^4 \dxy \lesssim \frac{L_0^2}{L^4} (|\ln \delta|^4 + 1). \label{24.1b}
 \end{equation}
\end{lemma}

\begin{proof}
To estimate $u_{3,x}^4$ we first estimate $\| u_{,x}(x,\cdot) \|_{L^{\infty}(-L_0,L_0)}$.  Since $u$ is an odd function of $y$, we know that $u_{,x}(x,0) = 0$. Then by H\"older's inequality
\begin{multline}\nonumber
 |u_{,x}(x,y)| = \left|\int_0^y u_{,xy}(x,y') \ud y'\right| \le |y|^{1/2} \left( \int_0^{y} u_{,xy}^2(x,y') \ud y'\right)^{1/2} \\ \le L_0 \left( \fint_0^{L_0} u_{,xy}^2(x,y') \ud y'\right)^{1/2} \overset{\eqref{reg4}}{\lesssim} L_0 (x^{-1} (|\ln x|^2+1))^{1/2},
\end{multline}
where we used that $u$ is $2L_0$-periodic, and so we could assume that $|y| \le L_0$. Similarly we get 
\begin{equation}\nonumber
 |u(x,y)| = \left|\int_0^y u_{,y}(x,y') \ud y'\right| \le L_0 \left( \fint_0^{L_0} u_{,y}^2(x,y') \ud y' \right)^{1/2} \overset{\eqref{AL}}{\lesssim} L_0 x^{1/2}.
\end{equation}
Since $u_{3,x} = \varphi_{\delta,x}u + \varphi_{\delta} u_{,x}$, we see that for $x \in [\delta/2,1]$:
\begin{equation}\nonumber
 |u_{3,x}(x,y)| \lesssim L_0 x^{-1/2} (|\ln x|^2 + 1)^{1/2},
\end{equation}
and $u_{3,x} = -$ for $x \in [-1,\delta/2)$. 
Then 
\begin{multline}\nonumber
 \int_{-1}^{1} \lint u_{3,x}^4(x,y) \dxy \le \int_{\delta/2}^1 \lint u_{3,x}^2(x,y) ||u_{3,x}(x,\cdot)||_{L^\infty}^2 \dxy 
\\ \lesssim \int_{\delta/2}^1 \lint u_{3,x}^2(x,y) L_0^2 x^{-1} (|\ln x|^2+1) \dx \dy
\\ \overset{\eqref{reg1},\eqref{reg2}}{\lesssim} L_0^2 \int_{\delta/2}^{1} x^{-1} (|\ln x|^3 + 1) \dx \lesssim L_0^2 (|\ln \delta|^4 + 1),
\end{multline}
which proves the lemma.
\end{proof}

\item To evaluate the second integral on~\eqref{deriv15a} we observe that since $w_{2,y} = u_{3,y} = 0$ for $x \in [-1,0]$, we have
\begin{equation}\label{est2}
 \int_{-1}^0 \fint_{-L}^L |w_{2,y} + u_{3,y}^2/2 -x |^2 \dxy = \int_{-1}^0 \fint_{-L}^L x^2 \dxy = 1/3.
\end{equation}

 \item Now we estimate~\eqref{deriv16a}:
\begin{multline}\label{est3}
\int_0^1 \fint_{-L}^L |w_{2,y} + u_{3,y}^2/2 -x |^2 \dxy =  \int_0^1 \fint_{-L}^L |(\varphi_\delta^2(x)x - u_{3,y}^2(x,y)/2) + u_{3,y}^2(x,y)/2 - x|^2 \dxy \\ \overset{\varphi_\delta=1 \textrm{ in } [\delta,1]}{=} \int_{0}^{\delta} |\varphi_{\delta}^2(x)x - x|^2 \dx \overset{0 \le \varphi \le 1}{\le} \int_0^\delta x^2 \dx = \delta^3/3. 
\end{multline}

\item The next term from $\E_L$ which we estimate (denoted by~\eqref{deriv17a}) has the form 
\begin{equation*}
 L^{-2} \int_{-1}^1 \lint \left|L^2 w_{1,y} + w_{2,x} + u_{3,x} u_{3,y}\right|^2/2 \dxy.
\end{equation*}
We observe that it follows from the definition of $w_1$ (see \eqref{defwu}) that this term 
completely vanishes. 

\item To estimate~\eqref{deriv18a} we combine
\begin{multline}\nonumber
 \int_{-1}^{1} \lint u_{3,x}^2(x,y) \dxy \le \int_{\delta}^1 \lint u_{,x}^2 \dxy + 2\int_{\delta/2}^{\delta} \lint \varphi_{\delta,x}^2 u^2 + u_{,x}^2 \dxy 
\\ \overset{\eqref{reg1},\eqref{reg2}}{\le} \int_{\delta/2}^1 \lint u_{,x}^2 \dxy + C\delta^{-2} \int_{\delta/2}^{\delta} x^2 (|\ln x|+1) \dx + C \int_{\delta/2}^{\delta} |\ln x|+1 \dx 
\\ \le \int_{0}^1 \lint u_{,x}^2 \dxy + C\delta (|\ln \delta| + 1)
\end{multline}
and
\begin{equation}\nonumber
 \int_{-1}^{1} \lint u_{3,yy}^2(x,y) = \int_{-1}^{1} \lint \varphi_\delta^2(x,y) u_{yy}^2(x,y) \le \int_{0}^{1} \lint u_{,yy}^2(x,y) \dxy
\end{equation}
to obtain
\begin{equation}\label{est4}
\begin{aligned}
 \int_{-1}^{1} \lint u_{3,x}^2(x,y) + u_{3,yy}^2(x,y) \dxy &\le \int_{0}^{1} \lint u_{,x}^2(x,y) + u_{,yy}^2(x,y) \dxy + C \delta (|\ln \delta|+1) 
\\ &\overset{\eqref{sigmae}}{\le} \sigma + \epsilon/2 + C\delta(|\ln \delta| + 1). 
\end{aligned}
\end{equation}

 \item Now we want to estimate the first part of~\eqref{deriv19a}. 
From~\eqref{defwu} and definition of $\varphi_\delta$ we see that
\begin{multline}\nonumber
 \int_{-1}^{1} \fint_{-L}^L u_{3,xy}^2(x,y) \dxy = \int_{-1}^{1} \fint_{-L}^L \left( \varphi_{\delta,x} u_{,y} + \varphi_\delta u_{,xy}\right)^2 \dxy \\ = \int_{\delta}^1 \fint_{-L}^L u_{,xy}^2 \dxy + \int_{\delta/2}^{\delta} \fint_{-L}^L \left( \varphi_{\delta,x} u_{,y} + \varphi_\delta u_{,xy}\right)^2 \dxy. 
\end{multline}
Since
\begin{equation}\nonumber
 \int_{\delta}^1 \fint_{-L}^L u_{,xy}^2 \dxy \overset{\eqref{reg4}}{\lesssim} \int_{\delta}^1 x^{-1}(|\ln x|^2 + 1) \dx \lesssim |\ln \delta|^3 + 1,
\end{equation} 
and
\begin{align}\nonumber
\int_{\delta/2}^{\delta} \fint_{-L}^L \left( \varphi_{\delta,x} u_{,y} + \varphi_\delta u_{,xy}\right)^2 \dxy &\lesssim \int_{\delta/2}^\delta \fint_{-L}^L \delta^{-2} u_{,y}^2 + u_{,xy}^2 \dxy 
\\ &\overset{\eqref{AL},\eqref{reg4}}{\lesssim} \int_{\delta/2}^\delta \delta^{-2} x + x^{-1} (|\ln x|^2+1) \dx \lesssim |\ln \delta|^3 + 1,\nonumber
\end{align}
 we obtain
\begin{equation}\label{est5}
  \int_{-1}^{1} \fint_{-L}^L u_{3,xy}^2(x,y) \dxy \lesssim |\ln \delta|^3 + 1.
\end{equation}

 \item
It remains to remains to estimate the second part of~\eqref{deriv19a}. We combine
\begin{equation}\nonumber
 \int_{\delta}^{1} \lint u_{3,xx}^2 \dxy = \int_{\delta}^{1} \lint u_{,xx}^2 \dxy \overset{\eqref{reg3}}{\lesssim} \int_\delta^1 x^{-2} (|\ln x|^7+1) \dx \lesssim \delta^{-1} (|\ln \delta|^7+1),
\end{equation} 
and
\begin{multline}\nonumber
 \int_{\delta/2}^{\delta} \lint u_{3,xx}^2 \dxy = \int_{\delta/2}^{\delta} \lint ( \varphi_{\delta,xx} u + 2 \varphi_{\delta,x} u_{,x} + \varphi_{\delta} u_{,xx})^2 \dxy \\
 \lesssim \int_{\delta/2}^{\delta} \lint \delta^{-4} u^2 + \delta^{-2} u_{,x}^2 + u_{,xx}^2 \dxy \overset{\eqref{reg1},\eqref{reg2},\eqref{reg3}}{\lesssim} \delta^{-1} (|\ln \delta|^7+1)
\end{multline}
to show
\begin{equation}\label{est6}
 \int_{-1}^{1} \lint u_{3,xx}^2 \dxy \lesssim \delta^{-1} (|\ln \delta|^7 + 1).
\end{equation}


\end{itemize}

We now put together estimates \eqref{24.1a},\eqref{24.1b}, and (\ref{est2}-\ref{est6}) to get:
\begin{align*}
\E_L(w,u_3) &= \int_{-1}^1 \lint \left( w_{1,x} + L^{-2} u_{3,x}^2/2 - 1 \right)^2 - 1 \dxy + \int_{-1}^0 \lint |w_{2,y} + u_{3,y}^2/2 - x|^2 \dxy 
\\ &\quad+ \int_{0}^1 \lint |w_{2,y} + u_{3,y}^2/2 - x|^2 \dxy 
\\ &\quad+ L^{-2} \int_{-1}^1 \lint \left|L^2 w_{1,y} + w_{2,x} + u_{3,x} u_{3,y}\right|^2/2 \dxy 
\\ &\quad + L^{-2} \int_{-1}^1 \lint u_{3,x}^2 + u_{3,yy}^2 \dxy 
\\ &\quad + L^{-4} \int_{-1}^1 \lint 2u_{3,xy}^2 + L^{-2} u_{3,xx}^2 \dxy 
\\ &\le \int_{-1}^{1} \lint -1 \dx \dy + \overbrace{\frac{1}{3}}^{\eqref{est2}} + \overbrace{L^{-2}(\sigma+\epsilon/2)}^{\eqref{est4}} + \overbrace{\frac{\delta^3}{3}}^{\eqref{est3}} 
\\ + \frac{C}{L^2}\Bigg( \big.&\underbrace{\frac{L_0^4}{L^2}\delta^{-1}(|\ln \delta|^8 + 1)}_{\eqref{24.1a}} + \underbrace{\frac{L_0^2}{L^2} (|\ln \delta|^4 + 1)}_{\eqref{24.1b}} + \underbrace{\delta (|\ln \delta| + 1)}_{\eqref{est4}} + \underbrace{\frac{|\ln \delta|^3+1}{L^2}}_{\eqref{est5}} + \underbrace{\frac{\delta^{-1}(|\ln \delta|^7+1)}{L^4}}_{\eqref{est6}} \Bigg).
\end{align*}
We set $\delta := 1/L$. Then $\delta^3/3 = 1/(3L^3)$ and the whole last line together with $\delta^3/3$ is bounded by $C(L_0^4 + 1)L^{-3}(\ln^8L + 1)$. Therefore 
\begin{equation}\nonumber
 \E_L(w,u_3) \le \mathcal{E}_0 + L^{-2} \left( \sigma + \epsilon/2 + C(L_0^4+1)L^{-1} (\ln ^8L + 1)\right)
\end{equation}
and 
\begin{equation}\nonumber
 \frac{1}{h} \left( \min_{(W,U_3)} E_h(W,U_3) - \mathcal{E}_0\right) = L^2 \left( \min_{(w,u_3)} \E_L(w,u_3) - \mathcal{E}_0 \right) \le \sigma + \epsilon/2 + C(L_\epsilon^4+1)L^{-1} (\ln ^8L + 1),
\end{equation}
where we used that $L_0 \le 2L_\epsilon$. Since the constant $C$ in the previous relation is universal and the function $L \mapsto L^{-1}(\ln^8 L + 1) \to 0$ as $L \to \infty$, we can find  $N_\epsilon$ large enough such that $C(L_\epsilon^4+1)L^{-1}(\ln^8 L + 1) \le \epsilon/2$ for $L \ge N_\epsilon L_\epsilon$. Such a choice of $N_\epsilon$ then implies
\begin{equation}\nonumber
 \frac{1}{h} \left( \min_{(W,U_3)} E_h(W,U_3) - \mathcal{E}_0\right) \le \sigma + \epsilon
\end{equation}
for any $h \in (0,(N_\epsilon L_\epsilon)^{-2})$, hence
\begin{equation}\nonumber
 \limsup_{h\to 0}\frac{1}{h} \left( \min_{(W,U_3)} E_h(W,U_3) - \mathcal{E}_0\right) \le \sigma + \epsilon.
\end{equation}
Since $\epsilon > 0$ was arbitrary, the proof of the upper bound is finished.

\section{Lower bound}\label{sect:lb}

In this section we will show the lower bound:
\begin{equation}\nonumber
 \liminf_{h\to0} \frac{\min_{(W,U_3)} E_h(W,U_3) - \mathcal{E}_0}{h} \ge \sigma.
\end{equation}
Based on the previous heuristic arguments (given in Section~\ref{heuristics}), we expect that only few terms from $\E_L$ (hence also $E_h$) will play role. Indeed, we will prove that it is enough to consider the second term in~\eqref{deriv15}, and terms from~\eqref{deriv17} and \eqref{deriv18}.

Let $\epsilon > 0$ and $0 < h \le 1$ be fixed. Then there exists a deformation $(w,u_3)$ (obtained by rescaling of some $(W,U_3)$) such that 
\begin{equation}\label{lb49}
 R_L(w,u_3) \le \min_{(W,U_3)} E_h(W,U_3) + \epsilon,
\end{equation}
where as before $L = h^{-1/2} \ge 1$. 
Jensen's inequality together with periodicity of $w_2$ in $y$ imply
\begin{equation}\nonumber
\int_0^1 \lint \left| w_{2,y} + u_{3,y}^2/2 - x \right|^2 \dxy \ge \int_0^1 \left( \lint u_{3,y}^2(x,y)/2 \dy  - x \right)^2 \dx, 
\end{equation}
and
\begin{align*}
\int_{-1}^0 \lint \left| w_{2,y} + u_{3,y}^2/2 - x \right|^2 \dxy &\ge \int_{-1}^0 \left( \lint u_{3,y}^2(x,y)/2 \dy  - x \right)^2 \dx 
\\ &\ge \int_{-1}^0 \left[ x^2 + \left( \lint u_{3,y}^2(x,y)/2 \dy \right)^2 \right] \dx,
\end{align*}
where the last inequality follows from $\lint u_{3,y}^2(x,y)/2 \dy \ge 0$ and from $-x \ge 0$ for $x \in [-1,0]$. We drop several non-negative terms from (\ref{deriv15}-\ref{deriv19}) to get
\begin{align}\nonumber
 R_L(w,u_3) &\ge \int_{-1}^1 \lint - 1 \dxy + \int_{-1}^0 x^2 \dx 
 \\ &\quad +  L^{-2} \int_0^1 \lint u_{3,x}^2 + u_{3,yy}^2 \dxy + \int_{-1}^1 \left( \lint u_{3,y}^2(x,y)/2 \dy  - \vt(x) \right)^2 \dx\label{RLge}
\\ &= \mathcal{E}_0 + L^{-2} \S(u_3) + \int_{-1}^1 \left( \lint u_{3,y}^2(x,y)/2 \dy  - \vt(x) \right)^2 \dx,\nonumber
\end{align}
where 
\begin{equation}\label{upsilon}
 \vt(x) := \begin{cases} 0 & \quad x \in [-1,0]\\ x & \quad x \in [0,1]. \end{cases}
\end{equation}
To show the lower bound we will use the following lemma: 
\newsavebox{\lemboxA}\begin{lrbox}{\lemboxA}
\begin{minipage}{\textwidth}\setlength{\parindent}{\parindentsave}
\vspace{\abovedisplayskip}
\begin{lemma}\label{lm:deltal}
 For any $L \ge 1$ there exists $\delta_L > 0$ such that 
\begin{equation}\label{deltal}
 \sigma_L = \min_{u \in \A} \S(u) \le \min_{v} \left( \S(v) + L^2 \int_{-1}^1 \left( \lint v_{,y}^2(x,y)/2 \dy  - \vt(x) \right)^2 \dx\right) + \delta_L,
\end{equation}
where we minimize over all $v : [-1,1]\times[-L,L] \to \R$ which are $2L$-periodic in the $y$-variable, and $\vt(x) = x \chi_{[0,1]}(x)$ (see \eqref{upsilon}). Moreover, 
\begin{equation}\label{deltalto0}
 \delta_L \to 0 \textrm{ as } L \to \infty.
\end{equation}
\end{lemma}
\end{minipage}\end{lrbox}\newline\noindent\usebox{\lemboxA}

\medskip
We now finish the proof of the lower bound, while we postpone the proof of the lemma to the very end of the paper (see Section~\ref{pf:lm41}).
By Lemma~\ref{lm:deltal} we have for every $\epsilon > 0$:
\begin{align*}
 \sigma_L &= \min_{u \in \A} \S(u) \overset{\eqref{deltal}}{\le} \min_{v} \left( \S(v) + L^2 \int_{-1}^1 \left( \lint v_{,y}^2(x,y)/2 \dy  - \vt(x) \right)^2 \dx\right) + \delta_L 
  \\ &\le \S(u_3) + L^2 \int_{-1}^1 \left( \lint u_{3,y}^2(x,y)/2 \dy  - \vt(x) \right)^2 \dx  + \delta_L 
 \\ &\!\!\overset{\eqref{RLge}}{\le} L^2(R_L(w,u_3) - \mathcal{E}_0) + \delta_L \overset{\eqref{lb49}}{\le} \frac{1}{h} \left( \min_{(W,U_3)} E_h(W,U_3) - \mathcal{E}_0 + \epsilon \right) + \delta_L.
\end{align*}
We send $\epsilon \to 0$ (while holding $L^2 = h^{-1}$ fixed) to get
\begin{equation}\nonumber
 \frac{\min_{(W,U_3)} E_h(W,U_3) - \mathcal{E}_0}{h} \ge \sigma_L - \delta_L.
\end{equation}
Finally, \eqref{deltalto0} implies:
\begin{equation}\nonumber
 \liminf_{h \to 0} \frac{\min_{(W,U_3)} E_h(W,U_3) - \mathcal{E}_0}{h} \ge \liminf_{L \to \infty} \left(\sigma_L - \delta_L \right) \overset{\eqref{def:sigma},\eqref{deltalto0}}{=} \sigma.
\end{equation}

\section{Scalar variational problem (proof of Theorem~\ref{thm:u})}\label{sect:scalar}

In this section we prove Theorem~\ref{thm:u}, which amounts to study existence and properties of a global minimizer (ground state) $u$ of 
\begin{equation}\label{def:s}
 \S(u) = \int_0^1 \lint u_{,x}^2 + u_{,yy}^2 \dxy,
\end{equation}
where we minimize over all functions $u : [0,1] \times [-L,L] \to \R$, which are $2L$-periodic in the $y$-variable and satisfy the following constraint:
\begin{equation}\label{def:constrain}
\fint_{-L}^L u_{,y}^2(x,y) \dy = 2x, \qquad \textrm{ for a.e. } x \in [0,1].
\end{equation}
We point out that all the results should hold uniformly in $L \ge 1$. In particular, all the constants will be independent of $L \ge 1$. 

Throughout this section let $L \ge 1$ be fixed. Before we start with the proof of Theorem~\ref{thm:u}, let us outline the strategy of the proof. We first observe that both the energy and the constraint can be written in terms of $L^2$-norms of $u_{,x}$, $u_{,yy}$, $u_{,y}$, and that $u$ is a periodic function of $y$, and so we can write the energy and the constraint using Fourier coefficients of $u$ in the $y$-variable. We also show that we can assume that $u$ is an odd function in the $y$-variable, i.e. that it can be written as $u(x,y) = \sum_{k \in \Ip} a_k(x)\sin(ky)$. Then we show the existence of a minimizer, and some elementary properties of the coefficients $a_k$ (e.g. that for any $k \in \Ip$ either $a_k(x) > 0$ for $x \in (0,1]$ or $a_k \equiv 0$). 

As a next step we derive the Euler-Lagrange equation (with Lagrange multiplier being a measure)
together with some preliminary estimates (lower bounds) on the Lagrange multiplier. 
Then we introduce quantities $\mu_k$ defined by $a'_k(x) = \mu_k(x) a_k(x) k^2$. Using $\mu_k$ we replace the linear second-order ODE for $a_k$ by a first-order nonlinear ODE for $\mu_k$. Moreover, we derive a useful comparison principle between $\mu_k$ and $\mu_{k'}$, which will later provide a way to study the behavior of $\mu_k$. In particular, we will show that $\mu_k(x)$ has to stay close to $-1$ for ''most`` $x$ in the interval $(k^{-2},1)$, which in turn implies exponential decay of $a_k$. Finally, using decay of higher frequencies we show the desired regularity estimates on $u$. 

\subsection{Existence of a minimizer $u$ of $\S$}

In this part we first use Fourier series (in $y$) to rewrite $u$ and $\S$. Then we show that for $u' \in \A$ we can find $u \in \A$, an odd function in $y$, while not increasing the energy: $\S(u) \le \S(u')$. We also show the existence of a minimizer $u$ for $\S$ and some properties of $u$. 

The following lemma is a simple consequence of the Plancherel theorem:

\begin{lemma}[Fourier representation]
 Let
\begin{equation}\nonumber
  a_k(x) := \begin{cases} \displaystyle \frac{1}{L} \int_{-L}^L u(x,y) sin(ky) \dy,& \quad k \in \I, k > 0, x\in[0,1],
\\ \displaystyle  \frac{1}{L} \int_{-L}^L u(x,y) cos(ky) \dy,& \quad k \in \I, k \le 0, x \in [0,1].\end{cases}
\end{equation}
Then 
\begin{equation}\nonumber
 \S(u) = \int_0^1 \left( \sum_{k \in \I, k \neq 0}  a'^2_k(x) + \frac{a_0'^2(x)}{4} + \sum_{k \in \I} a_k^2(x)k^4 \right) \dx 
\end{equation}
and the constraint~\eqref{def:constrain} turns into
\begin{equation}\nonumber
 \sum_{k \in \I} a_k^2(x)k^2 = 2x, \qquad \textrm{for a.e. } x \in [0,1].
\end{equation}
\end{lemma}


\begin{remark}\label{rmkae}
 Since for every $k\in \I$ one has $a_k \in W^{1,2}(0,1)$, the quantity $\sum_{k \in \I} a_k^2(x) k^2 \in [0,\infty]$ is defined for every $x \in [0,1]$ (but can possibly be infinite). Using the Fundamental theorem of calculus we get for $x_0,x_1 \in [0,1]$ 
 \begin{equation}\nonumber
  \left| a_k^2(x_1)k^2 - a_k^2(x_0)k^2 \right| \le \int_{x_0}^{x_1} \left|a_k'^2(x)k^2\right| \dx \le \int_{x_0}^{x_1} \left| 2a_k(x)k^2 a'_k(x) \right| \dx \overset{Young's}{\le} \int_{x_0}^{x_1} a_k'^2(x) + a_k^2(x) k^4 \dx,
 \end{equation}
 and so 
 \begin{equation}\nonumber
  \left| \sum_{k \in \I} a_k^2(x_1)k^2 - \sum_{k \in \I} a_k^2(x_0)k^2 \right| \le \int_{x_0}^{x_1} \sum_{k\in\I} a_k'^2(x) + a_k^2(x) k^4 \dx.
 \end{equation}
 Then $\S(u) < \infty$ implies that $\sum_{k\in\I} a_k^2(\cdot)k^2$ belongs to $W^{1,1}(0,1)$, in particular it is a continuous function, and so $\sum_{k\in\I} a_k^2(x)k^2 = 2x$ holds for every (and not only a.e.) $ x \in [0,1]$. 
\end{remark}

 The following lemma, which will be used later several times to combine two different deformations to construct a third one, is based on the fact that the energy $\S$, written in the quadratic variables $a_k^2$, is convex:

\begin{lemma}\label{lm:sublinear}
 Let $\mathcal{I}$ be an (at most) countable subset of $\mathbb{R}$, and let $\{ a_k(x) \}_{k \in \mathcal{I}}$ and $\{ b_k(x) \}_{k \in \mathcal{I}}$ be two families of real functions defined on $[0,1]$. Then the family $\{ c_k \}_{k \in \mathcal{I}}$, defined by
 \begin{equation}\nonumber
  c_k(x) := \sqrt{ a_k^2(x) + b_k^2(x)},
 \end{equation}
 satisfies 
 \begin{equation}\nonumber
  \int_0^1 \sum_{k \in \mathcal{I}} c_k'^2(x) + c_k^2(x) k^4 \dx \le \int_0^1 \sum_{k \in \mathcal{I}} a_k'^2(x) + a_k^2(x) k^4 \dx + \int_0^1 \sum_{k \in \mathcal{I}} b_k'^2(x) + b_k^2(x) k^4 \dx
 \end{equation}
 and 
 \begin{equation}\nonumber
  \sum_{k \in \mathcal{I}} c_k^2(x) k^2 = \sum_{k \in \mathcal{I}} a_k^2(x) k^2 + \sum_{k \in \mathcal{I}} b_k^2(x) k^2.
 \end{equation}
 \end{lemma}

\begin{proof}
 Since $c_k^2(x) = a_k^2(x) + b_k^2(x)$, it is enough to show that
 \begin{equation}\nonumber
  \left[ \left( \sqrt{a_k^2(x) + b_k^2(x)} \right)'\right]^2 \le a_k'^2(x) + b_k'^2(x).
 \end{equation}
 By H\"older's inequality $a_k a'_k + b_k b'_k \le (a_k^2 + b_k^2)^{1/2} (a_k'^2 + b_k'^2)^{1/2}$, and so the chain rule implies
 \begin{equation}\nonumber
  \left[ \left( \sqrt{a_k^2 + b_k^2} \right) ' \right]^2 = \left( (a_k^2+b_k^2)^{-1/2} \left( a_ka'_k + b_k b'_k \right) \right)^2 \le a_k'^2 + b_k'^2.
 \end{equation}
\end{proof}

\begin{lemma}[Symmetry of the minimizer]\label{lm:odd}
 Let $u' \in \A$ has finite energy $\S(u') < \infty$. Then there exists $u \in \A$ such that $\S(\bar u) \le \S(u)$ and $u$ is an odd function of $y$.
\end{lemma}

\begin{proof}
 Let $a_k$ be the Fourier coefficients of $u'$. For $k \in \I, k > 0$ we define:
 \begin{equation}\nonumber
  c_k(x) := \left(a_k^2(x) + a_{-k}^2(x)\right)^{1/2}.
 \end{equation}
 By Lemma~\ref{lm:sublinear} $u(x,y) := \sum_{k\in\Ip} c_k(x) \sin(ky)$ satisfies $\S(u) \le \S(u')$, $\sum_{k\in\Ip} c_k^2(x) k^2 = 2x$, and $u$ is an odd function of $y$. 
\end{proof}

\begin{remark}
 By the previous lemma we can assume that $u \in \A$, a minimizer of $\S$, is an odd function of $y$, i.e it can be written in the form
 \begin{equation}\nonumber
   u(x,y) = \sum_{k\in\Ip} a_k(x) \sin(ky).
 \end{equation}
 Then
 \begin{equation}\nonumber
  \S(u) = \int_0^1 \sum_{k\in\Ip} a_k'^2(x) + a_k^2(x)k^4 \dx,
 \end{equation}
 and by Remark~\ref{rmkae} constraint \eqref{def:constrain} turns into 
 \begin{equation}\label{constr:f}
  \sum_{k\in\Ip} a_k^2(x) k^2 = 2x, \qquad \forall x \in [0,1].
 \end{equation}
 Moreover, we can assume that
 \begin{equation}\label{a:pos}
  a_k(x) \ge 0,\quad \forall k\in \Ip, x \in [0,1].
 \end{equation}
\end{remark}

Now we are ready to show the existence of a minimizer $u \in \A$ for the functional $\S$:

\begin{proposition}\label{prop:minimizer}
 There exists $u \in \A$ such that $\S(u) < \infty$ and $\S(u) \le \S(u')$ for any $u' \in \A$.
\end{proposition}

Before we prove the proposition, let us state and prove a simple lemma, which will be used for the construction of a function $u \in \A$ with finite energy:
\begin{lemma}\label{lm:construction}
 There exists a universal constant $C$ which satisfies the following. For any $0 < b \le 1$ there exists a function $u : [0,1]\times[-L,L] \to \R$, which is $2L$-periodic in $y$ and such that
 \begin{align}
  \fint_{-L}^L u_{,y}^2(x,y) \dy &= 2x, \qquad x \in [0,b],\label{construction:1}\\
  u(x,y) &= 0, \qquad x \in [4b,1],\ y \in [-L,L],\label{construction:2}\\
  \lint \left|\frac{\partial^{\alpha+\beta} u}{\partial_x^\alpha \partial_y^\beta}(x,y)\right|^2 \dy &  \le C x^{2-2\alpha-\beta},\quad \alpha \in \{0,1,2\},\ \beta \ge 0, \ x \in [0,1].\label{construction:3}
 \end{align}  
\end{lemma}

\begin{proof}
 Let $f : [0,\infty) \to [0,1]$ be a fixed $C^2$ function with $\supp f \subset [1/4,4]$ which for $t \in [1/4,1]$ satisfies:
 \begin{equation}\label{cond:f}
  f^2(t) + \frac{1}{4} f^2(4t) = 2t.
 \end{equation}
 To prove the existence of such a function $f$, we first write $f(t) = \sqrt{t}\phi(t)$ and look for a function $\phi(t)$ which for $t \in [1/4,1]$ satisfies $\phi^2(t) + \phi^2(4t)=2$. Consider a smooth function $\bar \phi \ge 0$ with $\supp \bar \phi \subset [1/4,4]$ such that $\bar \phi = 1$ in $[1/2,2]$, and define 
\begin{equation}\nonumber
 \phi(t) := \frac{\sqrt{2} \bar \phi(t)}{\left(\bar \phi(t)^2 + \bar \phi(4t)^2\right)^{1/2}}, \qquad \phi(4t) := \frac{\sqrt{2} \bar \phi(4t)}{\left(\bar \phi(t)^2 + \bar \phi(4t)^2\right)^{1/2}}, \qquad t \in [1/4,1],
\end{equation}
and by $0$ elsewhere. Since $\bar \phi(t)^2 + \bar \phi(4t)^2 \ge 1$ for $t \in [1/4,1]$, $\phi$ is smooth and satisfies the desired condition $\phi^2(t) + \phi^2(4t) = 2$, $t \in [1/4,1]$. 
  
 We use the function $f$ to define the Fourier coefficients $a_k$ of $u$. We set $P := \pi \frac{\lfloor L \rfloor}{L}$ (recall that $L \ge 1$), and let $n_0 \ge 0$ be such that $\frac{1}{4} \le 4^{n_0} b \le 1$. 
 Then for $n_0 \le n < \infty$ we define:
 \begin{equation}\nonumber
  a_{k_n}(x) := P^{-1} 4^{-n} f(4^n x), \quad k_n := 2^n P,
 \end{equation}
 and $a_k=0$ for all the remaining $k \in \Ip$. Then for $n_0 \le N < \infty$ and $x \in [\frac{1}{4}4^{-N},4^{-N}]$ we have 
 \begin{multline}\nonumber
  \sum_{k\in\Ip} a_k(x)^2 k^2 = \sum_{n \ge n_0} P^{-2} 4^{-2n} f^2(4^nx) \left(2^n P\right)^2 = 4^{-2N} f^2(4^N x) 2^{2N}+ 4^{-2(N+1)} f^2(4^{(N+1)}x) 2^{2(N+1)} \\ = 
  4^{-N} \left( f^2(4^N x) + \frac{f^2(4\cdot4^Nx)}{4}\right) \stackrel{\eqref{cond:f}}{=} 4^{-N} 4^N 2x = 2x,
 \end{multline}
 and~\eqref{construction:1} follows. Since $n_0$ is such that $a_k(x)=0$ for any $x \in [4b,1]$ and any $k \in \Ip$, relation \eqref{construction:2} holds. 

 It remains to show the estimate on the derivatives (see~\eqref{construction:3}). For $n_0 \le N < \infty$ and $x \in [\frac{1}{4}4^{-N},4^{-N}]$, boundedness of $f$, $f'$, and $f''$ implies 
 \begin{equation}\nonumber
\fint_{-L}^L \left|\frac{\partial^{\alpha+\beta} u}{\partial_x^\alpha \partial_y^\beta}(x,y)\right|^2 \dy = \sum_{k\in\Ip} \left( \frac{\partial^\alpha a_k}{\partial x^\alpha}(x)\right)^2 k^{2\beta} \lesssim 4^{-2N} 4^{2N\alpha} 4^{N\beta} \lesssim x^{2-2\alpha-\beta}.
 \end{equation}
\end{proof}

\begin{proof}[Proof of Proposition~\ref{prop:minimizer}]
 By applying Lemma~\ref{lm:construction} with $b=1$ we obtain a function $u \in \A$ such that $\S(u) < \infty$. 
 Indeed,~\eqref{construction:1} implies \eqref{def:constrain}, and by~\eqref{construction:3} 
 \begin{equation}\nonumber
  \fint_{-L}^L u_{,x}^2(x,y) + u_{,yy}^2(x,y) \dy \lesssim 1 \quad \forall x \in [0,1].
 \end{equation}
 Integrating the above relation in $x$ gives $\S(u) \lesssim 1$. 

 Now let $u_n \in \A$ be a minimizing sequence for $\S$. Then $u_n$ is bounded in $H^1(\Omega)$ and $u_{n,yy}$ is bounded in $L^2(\Omega)$. Passing to a subsequence one has
 \begin{equation}\label{prop31:1}
  \begin{array}{rcll}
   u_n &\rightharpoonup& u \quad &\textrm{in } H^1(\Omega),\\
   u_{n,yy} &\rightharpoonup& u_{,yy} \quad &\textrm{in } L^2(\Omega). 
  \end{array}
 \end{equation}
 By a standard lower semi-continuity result we have that $\S(u) \le \liminf_{n \to \infty} \S(u_n)$. Hence it remains to show that the limit $u \in \A$, in particular that it satisfies the constraint~\eqref{def:constrain}. This is a direct consequence of the following lemma:
 
 \begin{lemma*}[Lemma 2.2 in~\cite{bib-bob+muller1}]
  Let $\Omega \subset \R^2$ and let $u_n : \Omega \to \R$ be a sequence such that $u_n \rightharpoonup u$ in $H^1(\Omega)$ and such that $u_{n,yy}$ lies in a compact subset of $H^{-1}(\Omega)$. Then
  \begin{equation}\nonumber
   u_{n,y} \to u_{,y} \quad \textrm{in } L^2_{loc}(\Omega).
  \end{equation}
 \end{lemma*}

 Indeed, by~\eqref{prop31:1} $u_n$ converges weakly in $H^1(\Omega)$. Moreover, $u_{n,yy}$ is bounded in $L^2(\Omega)$, and so it lies in a compact subset of $H^{-1}(\Omega)$ (here we use that $L^2$ is compactly imbedded into $H^{-1}$, which follows by duality argument from the Rellich-Kondrachev Theorem). Then the lemma applies, in particular we obtain that for a.e. $x \in [0,1]$:
 \begin{equation}\nonumber
  \lint u_{,y}^2 \dy = 2x.
 \end{equation}
 This completes the proof of the proposition.  
\end{proof}

The next lemma shows that the equality in the constraint~\eqref{constr:f} can be relaxed. It will be used in the proof of the Euler-Lagrange equations and of the existence of the Lagrange multiplier. Later, it will help to simplify the construction of competitors for a minimizer $u$ of $\S$. 

\begin{lemma}[Relaxation of the constraint]\label{lm:relaxed:constrain}
We have
\begin{equation}\label{eq:relaxed}
 \min_{\sum_{k\in\Ip} a_k^2(x)k^2 = 2x} \S(u) =  \min_{\stackrel{\sum_{k\in\I} a_k^2(x)k^2 \ge 2x}{a_k(0)=0}} \S(u).
\end{equation} 
\end{lemma}

\begin{proof}
 First, following the proof of Proposition~\ref{prop:minimizer}, one can show the existence of a minimizer $u'$ of $\S(u')$ under the constraint $\sum_{k\in\I} a_k^2 k^2 \ge 2x$. Based on $u'$, let $b_k \ge 0$ be its Fourier coefficients and set $f(x) := \sum_{k\in\I}  b_k^2(x)k^2$. 

Let us assume that $f(x) > 2x$ in a set of positive measure. As in Remark~\ref{rmkae}, we observe that $f \in W^{1,1}(0,1)$ is a continuous functions. Hence $\left\{x: f(x) > 2x\right\}$ is a relatively open subset of $(0,1]$, and as such it can be written as a countable union of relatively open intervals. Consider first $(a,b) \subset (0,1)$ being one such interval. Then $f(a)=2a$ and $f(b)=2b$, and $f(x) > 2x$ for $x \in (a,b)$. We make a variation of $u'$ with a compact support in $(a,b)$ to derive the Euler-Lagrange equation for the coefficients of $u'$: 
 \begin{equation}\nonumber
  b_k''(x) = b_k(x)k^4 \ge 0, \quad x \in (a,b).
 \end{equation}
 Hence in the interval $(a,b)$, the coefficients $b_k$ are convex functions. Therefore $f$ is also a convex function in the interval $(a,b)$, which gives a contradiction with $f(a)=2a, f(b)=2b,$ and $f(x) > 2x$. 

 It remains to consider the case $(a,1]$. In this case we are missing the information $f(b)=2b$, and so we need to argue differently. We observe that in this case we obtain boundary conditions for $b_k$, more precisely there holds $b_k'(1) = 0$. Since, by the Euler-Lagrange equation, each $b_k$ is convex in the interval $(a,1)$, and its derivative vanishes at $1$, we obtain that each $b_k$ is a non-increasing function. Therefore $f$ is also non-increasing, in particular $f \le f(a) = 2a$ in $(a,1]$, a contradiction with $f(x) > 2x$ in $(a,1]$. 

 We have shown that $f(x) = 2x$ for $x \in [0,1]$, which concludes the proof. 
\end{proof}

\begin{cor}\label{cor:ak:positive}
 Let $u$ be an odd minimizer of $\S$ which satisfies the constraint~\eqref{def:constrain}. Then for each $k \in\Ip$ either
 \begin{equation}\nonumber
  a_k(x) > 0, \quad x \in (0,1],
 \end{equation}
 or
 \begin{equation}\nonumber
  a_k(x) = 0, \quad x \in [0,1].
 \end{equation}
\end{cor}

\begin{proof}
 We assume the contrary, i.e. that $a_k(x_0) = 0$ for some $x_0 \in (0,1]$ while $\sup_{x \in (0,x_0)} a_k(x) > 0$. Let $f_k(x) := e^{k^2(x-1)} + e^{-k^2(x-1)}$. 
 Since $a_k(x_0) = 0$ and $\sup_{x \in (0,x_0)} a_k(x) > 0$, we can find $\epsilon > 0$ such that $\epsilon f_k(x)$ intersects the graph of $a_k$ in $(0,x_0)$. Let $\alpha < x_0$ be the maximal point of intersection (i.e. $a_k(\alpha) = \epsilon f_k(\alpha)$, and $a_k < \epsilon f_k$ in $(\alpha,x_0)$). Since both $a_k$ and $f_k$ are continuous, such a point exists. In case $\epsilon f_k$ intersects $a_k$ in the interval $(x_0,1]$, we set $\beta \in (x_0,1]$ to be the minimal point of intersection. If $\epsilon f_k > a_k$ in $(x_0,1]$, we set $\beta = 1$. 

 We define a competitor $\bar u$, which is obtained from $u$ by replacing $a_k$ in the interval $(\alpha,\beta)$ with the function $\epsilon f_k$. Let $\bar a_k$ denote the Fourier coefficients of $\bar u$. Since $a_k < \epsilon f_k$ in $(\alpha,\beta)$, we have $a_k \le \bar a_k$, in particular $\sum_{k\in\Ip} \bar a_k^2(x) k^2 \ge \sum_{k\in\Ip} a_k^2(x) k^2 \ge 2x$. 

 Finally, we observe that $\S(\bar u) < \S(u)$. Indeed, since $f_k''(x) = k^4 f_k$ and $f_k'(1)=0$, function $\epsilon f_k$ is the unique absolute minimizer for the functional $\int_\alpha^\beta a_k'^2 + a_k^2k^4$ with a given boundary conditions $a_k(\alpha)$ (and possibly $a_k(\beta)$). Since $a_k(x_0) \neq \epsilon f_k(x_0)$ (i.e. they are not identical), we have $\S(\bar u) < \S(u)$. 

 By Lemma~\ref{lm:relaxed:constrain} we have $\min_{\sum a_k^2 k^2 = 2x} \S = \min_{\sum a_k^2 k^2 \ge 2x} \S$, and so we obtain a contradiction:
 \begin{equation*}
  \S(\bar u) < \S(u) = \min_{\sum a_k^2 k^2 = 2x} \S = \min_{\sum a_k^2 k^2 \ge 2x} \S \le \S(\bar u).
 \end{equation*}
\end{proof}

\subsection{Euler-Lagrange equation and Lagrange multiplier}

In this part we will first derive Euler-Lagrange equation for $a_k$ and show the existence of a {\it Lagrange multiplier} $\lambda$ (a non-negative measure on $(0,1]$). As a next step, we will obtain some preliminary estimates on $\lambda$. Later in the section we will introduce and study quantities $\mu_k$.

\begin{lemma}[Euler-Lagrange equation and the Lagrange multiplier]\label{lm:lambda}
 Let $u \in \A$ be an odd minimizer of $\S$. Then there exists $\lambda$, a non-negative measure on $(0,1]$, such that for all $k \in \Ip$ and any $\varphi_k \in \mathcal{D}((0,1])$ 
 \begin{equation}\label{EL}
  \int_0^1 a_k'(x) \varphi_k'(x) + a_k(x)\varphi_k(x) k^4 = \int_0^1 a_k(x) \varphi_k(x) k^2 \ud \lambda(x). 
 \end{equation}
\end{lemma}

\begin{proof}
 Choose $K \in \Ip$ such that $a_K \not\equiv 0$. Then for any $\delta \ge 0$ and any test function $\varphi \in \mathcal{D}((0,1])$, $\varphi \ge 0$, we define
 \begin{equation}\nonumber
  b_k^{\delta}(x) := \begin{cases} a_k(x) & k \neq K\\ a_K(x) + \delta \varphi(x) & k = K, \end{cases}
 \end{equation}
 and by $v^\delta$ we denote the function with Fourier coefficients $b_k^\delta$. Since $\sum_{k\in\Ip} b_k^2(x) \ge \sum_{k \in \Ip} a_k^2(x) = x$,  Lemma~\ref{lm:relaxed:constrain} implies that $\S(v^\delta) \ge \S(u)$ for all $\delta \ge 0$, in particular
 \begin{equation}\nonumber
  0 \le \frac{1}{2} \left. \frac{d}{d\delta} \S(v^\delta)\right|_{\delta =0} = \int_0^1 a_K'(x) \varphi'(x) + a_K(x) \varphi(x)K^4 \dx. 
 \end{equation}
 Since $\varphi$ was arbitrary non-negative test function, we see that $-a_K''(x) + a_K(x)K^4$ is a non-negative distribution. Now we use that any non-negative distribution is a non-negative measure (see, e.g., \cite[Chapter 6.4]{strichartz:distributions}). Moreover, by Corollary~\ref{cor:ak:positive} $a_K > 0$ in $(0,1]$, and since $a_K \in W^{1,2}(0,1)$ is a continuous function in $[0,1]$, we have that $\inf a_K > 0$ on any compact subset of $(0,1]$. Hence we can find (locally finite) non-negative measure $\lambda$ in $(0,1]$ such that 
 \begin{equation}\label{ELK}
 \int_0^1 a_K'(x) \varphi'(x) + a_K(x)K^4 \varphi(x) \dx = \int_0^1 a_K(x)K^2 \varphi(x) \ud \lambda 
 \end{equation}
 for any $\varphi \in \mathcal{D}((0,1])$. We proved~\eqref{EL} for $k=K$. To extend~\eqref{EL} to all $l \in \Ip$, $a_l \not\equiv 0$, let us define for any $\delta \in \R$ and $\tilde \varphi \in \mathcal{D}((0,1])$ the following coefficients
 \begin{equation}\nonumber
  b_k^{\delta}(x) := \begin{cases} a_k(x) & k \neq K,l\\ a_K(x) + \delta a_l(x)l^2 \tilde \varphi(x) & k = K \\ a_l(x) - \delta a_K(x)K^2 \tilde \varphi(x) & k=l. \end{cases}
 \end{equation}
 Then we see that 
 \begin{equation}
  \sum_{k\in\Ip} b_k^2(x)k^2 = \sum_{k\in\Ip} a_k^2(x)k^2 + \delta^2\left( a_l^2(x)l^4 + a_K^2(x)K^4 \right) \tilde \varphi^2(x) \ge x, 
 \end{equation}
 and similarly as before we obtain (but this time using $\delta \in \R$)
 \begin{equation}\label{eq59}
  \int_0^1 a_K' \left( a_l \tilde \varphi \right)'l^2 + a_K a_l \tilde \varphi l^2 - a_l' \left( a_K \tilde \varphi  \right)'K^2 + a_l a_K \tilde \varphi K^2 = 0.
 \end{equation}
 We observe that by density argument~\eqref{ELK} holds for all $\varphi \in W^{1,2}(0,1)$ with compact support in $(0,1]$, in particular for $\varphi(x) := a_l(x) l^2 \tilde \varphi(x)$. Subtracting~\eqref{ELK} with $\varphi(x) = a_l(x) l^2 \tilde \varphi(x)$ from~\eqref{eq59} gives (after some algebraic manipulation)
 \begin{equation}\nonumber
  \int_0^1 a_l'(a_K \tilde \varphi)' + a_l l^4 (a_K \tilde \varphi) = \int_0^1 a_l l^2 (a_K \tilde \varphi) \ud \lambda.
 \end{equation}
 By observing that the previous relation holds for larger class of test functions $\tilde \varphi$ we obtain~\eqref{EL} for $k=l$. 
\end{proof}

Since the frequencies $a_k$, which are defined as elements of $W^{1,2}(0,1)$, satisfy~\eqref{EL}, the first derivative $a_k'(x)$ of a particular, say left-continuous representative $a_k$, has a well-defined value at every point $x \in (0,1]$. Moreover, as a consequence of~\eqref{EL} we obtain
\begin{cor}
Let $u \in \A$ be an odd minimizer of $\S$. Then there exists a number $\mu = \lambda(\{1\}) \ge 0$ such that for all $k \in \Ip$ 
 \begin{equation}\label{mu}
  a_k'(1) = \mu a_k(1) k^2,
 \end{equation} 
and for any $[\alpha,\beta] \subset (0,1]$ we have
\begin{equation}\label{ELint}
 a_k'(\alpha) - a_k'(\beta) = \int_{[\alpha,\beta)} a_kk^2 \ud \lambda - \int_\alpha^\beta a_k k^4 \dx. 
\end{equation}
\end{cor}

In the following lemma we show that $\lambda \ge 1/x$ in some sense, which will be important later in the proof of the exponential decay of $a_k$:

\begin{lemma}[Estimates on the Lagrange multiplier]\label{lm:lambda:lowerbound}
 Let $u \in \A$ be an odd minimizer of $\S$, 
and let $\lambda$ be the Lagrange multiplier obtained in Lemma~\ref{lm:lambda}. Then 
 \begin{equation}\label{lambda:ge}
  \lambda((\alpha,\beta)) \ge \int_\alpha^\beta \frac{1}{x} \dx
 \end{equation}
 for any $[\alpha,\beta] \subset (0,1)$ and 
 \begin{equation}\label{int:lambdax}
  \int_0^1 2x \ud \lambda = \S(u).
 \end{equation}
\end{lemma}

\begin{proof}
 For $K \ge 0$ and $\Psi \in \mathcal{D}(0,1)$ we define 
\begin{equation}\nonumber
  \varphi_k(x) := \begin{cases} a_k(x)k^2 \Psi(x) & k \in \Ip \cap [0,K]\\ 0 & \textrm{otherwise}. \end{cases}
 \end{equation}
 Then we sum~\eqref{EL} for $k \in \Ip \cap [0,K]$ to obtain
 \begin{equation}\label{64}
  \int_0^1 \sum_{k \in \Ip \cap [0,K]} (a_k'^2 k^2 + a_k^2 k^6) \Psi + \sum_{k \in \Ip \cap [0,K]} a_k'a_k k^2 \Psi' \dx = \int_0^1 \sum_{k \in \Ip \cap [0,K]} a_k^2k^4 \Psi \ud \lambda. 
 \end{equation}
 Now we observe that by differentiating the constraint~\eqref{constr:f}\footnote{The differentiation is justified by the fact that $\sum_{k\in\Ip} a_k'(\cdot)a_k(\cdot)k^2 \le \frac{1}{2} \sum_{k\in\Ip} a_k'^2(\cdot) + a_k^2(\cdot)k^4$, where the right-hand side is in $L^1(0,1)$.} we obtain $\sum_{k\in\Ip} a_k'a_kk^2 = 1$, which can be integrated against $\Psi'$ to show
 \begin{equation}\nonumber
  \int_0^1 \sum_{k\in\Ip} a_k'a_k k^2 \Psi' \dx = 0.
 \end{equation}
 Hence taking $K \to \infty$ in~\eqref{64}, and using monotone convergence theorem together with the previous relation imply
 \begin{equation}\label{66}
  \int_0^1 \sum_{k \in \Ip} (a_k'^2 k^2 + a_k^2 k^6) \Psi \dx = \int_0^1 \sum_{k \in \Ip} a_k^2k^4 \Psi \ud \lambda.  
 \end{equation}

  We now estimate from below both terms on the left-hand side of the above equation. By the Cauchy-Schwarz inequality
 \begin{equation}\nonumber
  \Bigg( \sum_{k\in\Ip} a_k^2 k^4 \Bigg)^2 \le \Bigg( \sum_{k\in\Ip} a_k^2 k^2 \Bigg) \Bigg( \sum_{k\in\Ip} a_k^2 k^6 \Bigg) \overset{\eqref{constr:f}}{=} 2x \sum_{k\in\Ip} a_k^2 k^6.
 \end{equation}
 For the first term in~\eqref{66} we differentiate the constraint~\eqref{constr:f} and use the Cauchy-Schwarz inequality
 \begin{multline}\nonumber
  2 = (2x)' = \Bigg( \sum_{k\in\Ip} a_k^2 k^2 \Bigg)' = 2 \sum_{k\in\Ip} a'_k a_k k^2 
\\ \le 2\Bigg( \sum_{k\in\Ip} a_k^2 k^2 \Bigg)^{1/2} \Bigg( \sum_{k\in\Ip} a_k'^2 k^2 \Bigg)^{1/2} = 2(2x)^{1/2} \Bigg( \sum_{k\in\Ip} a_k'^2 k^2 \Bigg)^{1/2}.
 \end{multline}
 The above estimates together with the Young's inequality imply
 \begin{equation}\nonumber
  \sum_{k \in \Ip} a_k'^2(x) k^2 + a_k^2(x) k^6 \ge \frac{1}{x} \sum_{k \in \Ip} a_k^2(x)k^4.
 \end{equation}
 Using this in~\eqref{66} gives
 \begin{equation}\nonumber
  \int_0^1 B(x) \Psi(x) \ud \lambda \ge \int_0^1 \frac{1}{x} B(x) \Psi(x) \dx,
 \end{equation}
 where we used notation 
 \begin{equation}\label{def:B}
  B(x) := \sum_{k\in\Ip} a_k^2(x) k^4.  
 \end{equation}
Since $B(x) > 0$ for $x \in (0,1]$ (otherwise~\eqref{constr:f} would be false) and $B \in L^1(0,1)$, by the Radon-Nikod\'ym Theorem we have for $[\alpha,\beta] \subset (0,1)$
 \begin{equation}\nonumber
  \lambda((\alpha,\beta)) \ge \int_\alpha^\beta \frac{\dx}{x}.  
 \end{equation}

To prove~\eqref{int:lambdax}, we test~\eqref{EL} with $\varphi_k := a_k$ and sum in $k$ to show that 
\begin{equation}\nonumber
 \S(u) = \int_0^1 \sum_{k \in \Ip} a_k'^2 + a_k^2k^4 \dx = \int_0^1 \sum_{k \in \Ip} a_k^2 k^2 \ud \lambda \overset{\eqref{constr:f}}{=} \int_0^1 2x \ud \lambda.
\end{equation}

%
\end{proof}

The Euler-Lagrange equation~\eqref{EL} consists of a linear homogeneous second order ODE plus the boundary conditions which are also homogeneous. In particular, they define $a_k$ only up to a multiplication by a constant. To remove this degree of freedom and to replace a linear second order ODE by a nonlinear first order ODE, we introduce the following quantity (for all $k \in\Ip$ such that $a_k \not \equiv 0$):
\begin{equation}\label{eq:mu}
 \mu_k(x) := \frac{a'_k(x)}{k^2 a_k(x)}.
\end{equation}
A simple computation shows that~\eqref{EL} and~\eqref{mu} translates into
\begin{equation}\label{ELmu}
\begin{aligned}
 \mu_k'(x) &= k^2(1-\mu_k^2(x)) - \lambda\\
 \mu_k(1) &=\mu = \lambda(\{1\}),
\end{aligned} 
\end{equation}
where the first equality holds in the sense of distributions. Since $\mu_k'$ is a measure, we can consider a particular (left-continuous) representative of $\mu_k$ which is defined for every $x \in (0,1]$ by
\begin{equation}\label{mu:repre}
 \mu_k(x) = \lambda([x,1]) - \int_x^1 k^2(1-\mu_k(x')) \ud x'.
\end{equation}
This is consistent with our previous choice of a left-continuous representative for $a_k$. 


In the following we will obtain an upper bound for the Lagrange multiplier $\lambda$. As a first step, we derive some estimates on $\mu_k$, which in turn will imply the exponential decay of $a_k$ and finally the desired estimate on $\lambda$. 

\begin{cor}\label{cor:muk:1}
 Let $u \in \A$ be an odd minimizer of $\S$, and let $k \in\Ip$ be such that $a_k \not \equiv 0$. Then 
 \begin{equation}\nonumber
  -1 < \mu_k(x) < 1, \quad \forall x \in [1/k^2,1].
 \end{equation}
\end{cor}

\begin{proof}
 It is enough to prove the following three claims.
 \begin{itemize}
  \item[1)] $\mu_k(k^{-2}) < 1$,
  \item[2)] $\mu_k(x) > -1$ for any $x \in (0,1]$,
  \item[3)] if $\mu_k(x_0) < 1$, then $\mu_k<1$ in $[x_0,1]$. 
 \end{itemize}

{\bf Claim 1:} We first observe that $a_k'$ is strictly decreasing function in $(0,k^{-2})$. Indeed, let $[\alpha,\beta] \subset (0,k^{-2})$. Then~\eqref{ELint} and~\eqref{lambda:ge} imply
\begin{equation}\label{75}
 a_k'(\alpha) - a_k'(\beta) + \int_\alpha^\beta a_k k^4 \dx = \int_\alpha^\beta a_k k^2 \ud \lambda \ge \int_\alpha^\beta a_k k^2 \frac{\dx}{x} \overset{k^2 < x^{-1}}{>} \int_\alpha^\beta a_k k^4 \dx.
\end{equation}

Then $a_k(0)=0$ and $a_k(x) > 0$ in $(0,1]$ imply
 \begin{equation}\nonumber 	
  a_k(k^{-2}) = \int_0^{1/k^2} a_k'(x) \dx > \int_0^{1/k^2} a_k'(1/k^2) \dx = k^{-2} a_k'(k^{-2}) = \mu_k(k^{-2}) a_k(k^{-2}),
 \end{equation}
 which shows that $\mu_k(k^{-2}) < 1$. 

{\bf Claim 2:}
 Let us assume $\mu_k(x) \le -1$ for some $x \in (0,1]$, i.e. $M := \left\{ x \in (0,1] : \mu_k(x) \le -1\right\}$ is not empty. 
 Then it follows from \eqref{mu:repre} (in particular the left-continuity of $\mu_k$) and from $\mu_k(1) = \mu = \lambda(\{1\}) \ge 0$ that we can find $x_0 \in (0,1)$, the maximal element of $M$. 
 Using~\eqref{ELmu} and Lemma~\ref{lm:lambda:lowerbound}, we see that $\mu_k'(x_0) < 0$, which means that $\mu_k \le -1$ in some right neighborhood of $x_0$, a contradiction with the definition of $x_0$. 

{\bf Claim 3:} If $x_0$ is such that $\mu_k(x_0) = 1$, then~\eqref{ELmu} and Lemma~\ref{lm:lambda:lowerbound} imply that $\mu_k'(x_0) < 0$, and so $\mu_k > 1$ in a left neighborhood of $x_0$. This part then follows directly from this observation.
 \end{proof}

Definition of $\mu_k$ implies that
\begin{equation}\label{66-}
 a_k(x_1) = a_k(x_0) e^{k^2 \int_{x_0}^{x_1} \mu_k(x) \dx},
\end{equation}
and so $a_k$ should decay exponentially provided $\mu_k$ stays well below $0$. Since $\S(u) < \infty$, \eqref{int:lambdax} provides an upper bound on the Lagrange multiplier $\lambda$. More precisely, for every $\epsilon \in (0,1)$ we have
\begin{equation}\nonumber
 \lambda((\epsilon,1]) \le \frac{\S(u)}{2\epsilon}.
\end{equation}
For large $k > 0$ we expect that $\mu_k$ will stay close to $\{-1,1\}$ most of the time, since otherwise by~\eqref{ELmu} $\lambda$ will often need to be ``large'' to balance the $k^2(1-\mu_k^2)$ term. Hence, to show the exponential decay of $a_k$ we need to rule out the case that $\mu_k$ stays close to $1$. To make these ideas rigorous we need some preparation. 

First, in the following lemma we will show that if $k > m$, then $\mu_k < \mu_m$ as long as $\mu_k < 1$, which by Corollary~\ref{cor:muk:1} holds at least in the interval $[1/k^2,1]$. Then we show that each $\mu_k$ is bounded from above by a function $f_k$, a solution to $f_k'(x) = k^2(1-f_k^2(x)) - 1/x$ with the initial condition $f_k(x) \to \infty$ as $x \to 0+$. It is easy to observe that $f_k$ can be obtained from $f_m$ by rescaling in $x$ by $m^2/k^2$, and so we just need to study $f_1$. We will show in Lemma~\ref{lm:faway1} that $f_1 < 1-\delta$ on a non-trivial interval. Hence, for $x > 1/k^2$, $\mu_k < \mu_m < f_m$ for any $m < k$, $a_m \not\equiv 0$, and so we get that $\mu_k$ stays away from $1$ provided we can show that there are ``many'' non-zero frequencies $a_m$. Since we expect that $\mu_k$ should stay close to $\{-1,1\}$ most of the time, the only possibility is that it stays close to $-1$, which would imply exponential decay of $a_k$. 

\begin{lemma}[Monotonicity of $\mu$'s]\label{lm:muk:monotone}
Let $u \in \A$ be an odd minimizer of $\S$, and $\{a_k\}_{k \in \Ip}$ be the corresponding Fourier coefficients. Let $k,m\in\Ip$, $k > m $ be such that $a_k \not \equiv 0$ and $a_m \not \equiv 0$. Then 
 \begin{equation}\nonumber
  \mu_k(x) < \mu_m(x), \quad \forall x \in [1/k^2,1).
 \end{equation}
\end{lemma}

\begin{proof}
 Subtracting~(\ref{ELmu}) for $\mu_m$ from (\ref{ELmu}) for $\mu_k$ gives
 \begin{equation}\nonumber
  \left(\mu_k - \mu_m\right)' = \left(k^2 - m^2\right) \left( 1 - \mu_k^2 \right) + m^2 \left( \mu_m^2 - \mu_k^2 \right),
 \end{equation}
 and $\mu_k(1) - \mu_m(1) = \mu - \mu = 0$. We see that $\mu_k - \mu_m$ is a continuous function, and so $M := \left\{ x : \mu_k(x) < \mu_m(x) \right\}$ is an open set . We want to show that $[1/k^2,1) \subset M$. First, by the previous lemma $\mu_k^2 < 1$ in $[1/k^2,1]$. Then $\mu_m(1) = \mu_k(1)$ together with $k > m$ imply $(\mu_k-\mu_m)'(1) > 0$, in particular some open left neighborhood of $1$ belongs to $M$. Let $x_0$ be the smallest point in $[0,1]$ such that $(x_0,1) \subset M$. If $x_0 < k^{-2}$, the lemma follows immediately. Let us therefore assume that $x_0 \in [1/k^2,1)$. 
 Since $\mu_k(x_0) = \mu_m(x_0)$, $k > m$, and $\mu_k^2(x_0) < 1$, we have that $\mu_k'(x_0) - \mu_m'(x_0) > 0$. In particular, $\mu_k > \mu_m$ in some right neighborhood of $x_0$, which contradicts the definition of $x_0$. 
\end{proof}

\begin{lemma}\label{lm:faway1}
 Let $f : [1,b] \to (-1,\infty)$ with some $b > 1$, $f(1) < 1$, be a left-continuous function which for any $[\alpha,\beta] \subset [1,b]$ satisfies
 \begin{gather}
   f(\beta) - f(\alpha) \le \int_{\alpha}^{\beta} 1 - f^2(x) - 1/x \dx\label{f:ode}.
 \end{gather}  
 Then $f(x) \le \sqrt{1-1/x}$ for all $x \in [4,b]$. 
\end{lemma}

\begin{proof}
 First we claim that $f(x) \le \sqrt{3/4}$ for some $x \in [1,2]$. Indeed, if this were not true, then $1-f^2(x)-1/x \le -1/4$ for $x \in [1,2]$, and so we would obtain a contradiction since~\eqref{f:ode} implies $\displaystyle f(2) \le f(1) + \int_1^2 -1/4 = 3/4 < \sqrt{3/4}$. Hence there exists $x_1 \in [1,2]$ such that $f(x_1) \le \sqrt{3/4}$. 

 We define $M := \left\{ x \in [1,b] : f(x) \le \sqrt{1-1/x} \right\}$. We claim that $x_2 \in M$ for some $x_2 < 4$, and that $[x_2,b] \subset M$. Since $f(x_1) \le \sqrt{3/4}$ and $1-f^2(x)-1/x < 0$ as long as $x \not\in M$, we see that $f(x) < f(x_1)$ for all $x > x_1$ such that $(x_1,x) \cap M = \emptyset$. Hence either $x \in M$ for some $x_1 < x < 4$ or $f(4) < \sqrt{3/4}$. In both cases we proved that $x \in M$ for some $x_1 < x \le 4$. 

 To prove that $[x_2,b] \subset M$, let us assume the contrary. Then there exists a maximal $x_3 \in [1,b)$ such that $[x_2,x_3] \subset M$ ($f$ is lower semicontinuous, and so such $x_3$ exists), and also $[x_3,x_3+\epsilon] \cap M = \emptyset$ for some $\epsilon > 0$, and $f(x_3) = \sqrt{1-1/x_3}$. Since for $x \in (x_3,x_3+\epsilon)$ we have $1-1/x-f^2(x) < 0$, for any $x_4 \in (x_3,x_3+\epsilon)$ \eqref{f:ode} implies:
 \begin{equation}
  f(x_4) \le f(x_3) + \int_{x_3}^{x_4} 1-1/x-f^2(x) \dx < f(x_3) = \sqrt{1-1/x_3} < \sqrt{1-1/x_4}.
 \end{equation}
 i.e. $x_4 \in M$ -- a contradiction.
 \end{proof}

 

\begin{cor}\label{cor:muk:bound}
 Let $l > 4$. Then there exists $\Delta=\Delta(l)$, $0 < \Delta \le 1/2$ such that for any $k > 0$, $a_k \not\equiv 0$ we have
 \begin{equation}\label{muk:1delta}
  \mu_k(x) \le 1-\Delta, \quad x \in [4/k^2,\min(l/k^2,1)].
 \end{equation}
\end{cor}

\begin{proof}
 If $k\le 2$, \eqref{muk:1delta} is trivially satisfied. If $k > 2$, let $f(x) := \mu_k(x/k^2)$, and observe that such $f$ satisfies all the assumption of Lemma~\ref{lm:faway1}. Indeed, $f(1) = \mu_k(k^{-2}) < 1$ by Corollary~\ref{cor:muk:1}, $\mu_k$ is left-continuous, and~\eqref{mu:repre} together with \eqref{lambda:ge} imply~\eqref{f:ode}.
 To conclude apply the previous lemma to get $f(x) \le \sqrt{1-1/x} \le \sqrt{1-1/l} = 1-\Delta$ for $x \in [4,l]$ and $\Delta = 1-\sqrt{1-1/l}$.
\end{proof}

In what follows we will show an upper bound for the {\it bending} part of the energy. More precisely, the following lemma will be used to prove that $\displaystyle\fint_0^{x_0} \sum_{k\in\Ip} a_k^2(x)k^4 \dx \lesssim 1$ for any $x_0 \in (0,1]$:

\begin{lemma}[Estimate on the average bending]\label{lm:bending:estimate}
 Let $u \in \A$ be an odd minimizer of $\S$. Then there exists a universal constant $C_1 < \infty$ such that for any $x_0 \in (0,1]$ 
 \begin{equation}\nonumber
  \int_0^{x_0} B(x) \dx \le \frac{B(x_0)}{3}x_0 + C_1 x_0,
 \end{equation} 
where $B(x) = \sum_{k\in\Ip} a_k^2(x) k^4$. 
\end{lemma}

\begin{proof}
 We prove the lemma by constructing a competitor for $u$. More precisely, we take $u$ and in the interval $[0,x_0]$ we replace its coefficients $a_k$ by linear functions. 
 
 Given $x_0 \in (0,1]$, we define
 \begin{equation}\nonumber
  \tilde a_k(x) := \begin{cases} a_k(x_0)\frac{x}{x_0} & x \in [0,x_0] \\ a_k(x) & x \in (x_0,1]. \end{cases}
 \end{equation}
 Before comparing the energy of the new deformation $\tilde a_k$ with the energy of $a_k$, we observe that the deformation $\tilde a_k$ does not satisfy the constraint~\eqref{constr:f}. Indeed, for $x \in (0,x_0)$
 \begin{equation}\label{eq:b2}
  \sum_{k\in\Ip} \tilde a_k(x)^2 k^2 = \sum_{k\in\Ip} a_k(x_0)^2 x^2 x_0^{-2} k^2 \overset{\eqref{constr:f}}{=} 2x_0 \frac{x^2}{x_0^2} < 2x. 
 \end{equation}
 To compensate this loss we use Lemma~\ref{lm:construction} with $a=0$ and $b=x_0$ to obtain $b_k$ such that
 \begin{equation}\nonumber
  \sum_{k\in\Ip} b_k^2(x) k^2 = 2x, \quad x \in [0,x_0],
 \end{equation}
 and
 \begin{equation}\nonumber
  \int_0^1 \sum_{k\in\Ip} b_k'^2 + b_k^2 k^4 \dx \le C_1 x_0,
 \end{equation}
 where $C_1$ is some universal constant. Now we combine $\tilde a_k$ and $b_k$ by setting $c_k^2(x) := \tilde a_k^2(x) + b_k^2(x)$ to obtain a deformation $c_k$ which satisfies
 \begin{gather}\nonumber
  \sum_{k\in\Ip} c_k^2(x) k^2 \ge 2x, \quad x \in [0,1],\\
  \int_0^1 \sum_{k\in\Ip} c_k'^2 + c_k^2 k^4 \dx \overset{\textrm{Lm}~\ref{lm:sublinear}}{\le} \int_0^1 \sum_{k\in\Ip} \tilde a_k'^2 + \tilde a_k^2 k^4 \dx + C_1 x_0.\label{eq:b3}
 \end{gather}
Since $u$ is a minimizer and $\sum_{k\in\Ip} c_k^2(x) k^2 \ge 2x$,  Lemma~\ref{lm:relaxed:constrain} and~\eqref{eq:b3} imply
 \begin{equation}\label{eq:b4}
  \int_0^1 \sum_{k\in\Ip} a_k'^2 + a_k^2 k^4 \dx \le \int_0^1 \sum_{k\in\Ip} c_k'^2 + c_k^2 k^4 \dx \overset{\eqref{eq:b3}}{\le} \int_0^1 \sum_{k\in\Ip} \tilde a_k'^2 + \tilde a_k^2 k^4 \dx + C_1 x_0.
 \end{equation}
 To conclude, we observe the following relations:
 \begin{align*}
  \int_{x_0}^1 \sum_{k\in\Ip} \tilde a_k'^2 + \tilde a_k^2 k^4 \dx &= \int_{x_0}^1 \sum_{k\in\Ip} a_k'^2 + a_k^2 k^4 \dx,\\
  \int_0^{x_0} \sum_{k\in\Ip} \tilde a_k^2(x) k^4 \dx &= \sum_{k\in\Ip} a_k^2(x_0) k^4 \int_0^{x_0} (x/x_0)^2 \dx = \frac{B(x_0)}{3} x_0,  \\
  \int_0^{x_0} \sum_{k\in\Ip} \tilde a_k'^2 \dx = \frac{1}{x_0} \sum_{k \in \Ip} a_k^2(x_0) &\le \int_0^{x_0} \sum_{k\in\Ip} a_k'^2 \dx,
 \end{align*}
 and so \eqref{eq:b4} implies
 \begin{equation}\nonumber
  \int_0^{x_0} B(x) \dx = \int_0^{x_0} \sum_{k\in\Ip} a_k^2 k^4 \dx \le \frac{x_0}{3}\sum_{k\in\Ip} a_k^2(x_0) k^4 + C_1 x_0 = \frac{x_0}{3} B(x_0) + C_1 x_0.
 \end{equation}  
\end{proof}

\begin{cor}\label{cor:bending}
 Let $u \in \A$ be an odd minimizer of $\S$. There exists a universal constant $C_2 < \infty$ such that 
 \begin{equation}\nonumber
  \int_0^{x_0} B(x) \dx \le C_2 x_0, \qquad x_0 \in (0,1].
 \end{equation}
\end{cor}

\begin{proof}
 Define $C_2 := \max\left( 6 C_1, 2\max_{l\ge1} \sigma_l\right) < \infty$, where $\sigma_l = \inf \mathcal{S}_l$ (see \eqref{sigmal}). We prove that if there exists $x_0$ such that $\displaystyle\int_0^{x_0} B(x) \dx \ge C_2 x_0$, then also $\displaystyle\int_0^1 B(x) \dx \ge C_2$. This would in turn give a contradiction since $\displaystyle\int_0^1 B(x) \dx \le \sigma_L < C_2$. 

 So let us assume that $\displaystyle\int_0^{x_0} B(x) \dx \ge C_2 x_0$ for some $x_0 \in (0,1]$. Then Lemma~\ref{lm:bending:estimate} implies that for any $x_1 \in [x_0,\min(2x_0,1)]$ we have
 \begin{equation}\nonumber
  C_2 x_0 \le \int_0^{x_0} B(x) \dx \le \int_0^{x_1} B(x) \dx \le \frac{B(x_1)}{3}x_1 + C_1 x_1 \le \frac{2B(x_1)}{3}x_0 + \frac{C_2}{3} x_0,
 \end{equation}
 and so $B(x_1) \ge C_2$ for any $x_1 \in [x_0,\min(2x_0,1)]$. Therefore
 \begin{equation}\nonumber
   \int_0^{x_1} B(x) \dx \ge \int_0^{x_0} B(x) \dx + (x_1 - x_0) C_2 \ge C_2 x_1.
 \end{equation}
 We proved that if $\displaystyle\int_0^{x_0} B(x) \dx \ge C_2 x_0$, then this is also true if we replace $x_0$ by any $x_1 \in [x_0,\min(2x_0,1)]$. Therefore it is also true for any $x_1 \in [x_0,1]$, in particular for $x_1 = 1$, which gives us a contradiction. 
\end{proof}

\subsection{Gap estimate}

In Corollary~\ref{cor:ak:positive} we showed that each frequency $a_k$ is either identically zero or is strictly positive in $(0,1]$, but so far we do not know which coefficients $a_k$ do not vanish. In this part we will estimate how large a gap between two non-vanishing frequencies can be.

\begin{lemma}\label{lm:trivak}
 Let $u \in \A$ be an odd minimizer of $\S$, and $\{a_k\}_{k \in \Ip}$ be the corresponding Fourier coefficients. Then for $k \in \Ip, k \le (1/2)^{1/4}$: 
 \begin{equation}\nonumber
  a_k \equiv 0.
 \end{equation} 
\end{lemma}

\begin{proof}
 First we claim that for any $k \in \Ip$ 
 \begin{equation}\label{membrane:le:bending}
  \int_0^1 a_k'^2(x) \dx \le 4 \int_0^1 a_k^2(x) k^4 \dx. 
 \end{equation}
 Indeed, if~\eqref{membrane:le:bending} were not true, then replacing $a_{k}$ and $a_{2k}$ by $\bar a_{k}:=0$ and $\bar a_{2k} := \sqrt{ a_{2k}^2 + a_{k}^2/4 }$ will decrease the energy while the constraint~\eqref{constr:f} will not change. Truly, we have $\bar a_{k}^2 k^2 + \bar a_{2k}^2 (2k)^2 = 0 + (a_{2k}^2 + a_{k}^2 / 4)(2k)^2 = a_{k}^2 k^2 + a_{2k}^2 (2k)^2$ and by Lemma~\ref{lm:sublinear} the energy decreases:
\begin{multline}\nonumber
 \int_0^1 \bar a_{k}'^2 + \bar a_{k}^2 k^4 + \bar a_{2k}'^2 + \bar a_{2k}^2 (2k)^4 \overset{Lm~\ref{lm:sublinear}}{\le} \int_0^1 \frac{ a_{k}'^2}{4} + 4 a_{k}^2 k^4 + a_{2k}'^2 + a_{2k}^2 (2k)^4 \dx 
 \\ < \int_0^1 a_{k}'^2 + a_{k}^2 k^4 + a_{2k}'^2 + a_{2k}^2 (2k)^4 \dx,
\end{multline}
where the last inequality holds provided~\eqref{membrane:le:bending} is false.

Since $a_k(0)=0$, by H\"older's inequality
\begin{equation}\nonumber
 \int_0^1 a_k^2(x) \dx = \int_0^1 \left( \int_0^x a'_k(x') \ud x' \right)^2 \le \left( \int_0^1 x \dx \right) \left( \int_0^1 a_k'^2(x) \dx \right) \overset{\eqref{membrane:le:bending}}{\le} 2k^4 \int_0^1 a_k^2(x) \dx.
\end{equation}
To conclude it is enough to observe that using the previous relation, $2k^4 < 1$ implies $\displaystyle\int_0^1 a_k^2(x) \dx = 0$.
\end{proof}

\begin{proposition}\label{prop:gap}
 There exists a universal constant $C_{gap}$ such that if $a_K \not\equiv 0$, $a_{GK}\not\equiv 0$, and $a_k \equiv 0$ for $k \in \Ip, K < k < GK$, then $G \le C_{gap}$. 
\end{proposition}
{
\providecommand{\bx}{\bar x}
\begin{proof}
 Let $0 < \epsilon < C_2/2$ and $\delta>0$ be fixed such that  $2\epsilon + 4\delta < 2-\sqrt{3}$. For the contrary we assume that 
 \begin{equation}\label{cgap}
  G > C_{gap} := \max\{ 4\sqrt{2C_2 \epsilon^{-1}}, \sqrt{24 \delta^{-1}} C_2 \epsilon^{-1} \}.
 \end{equation}
 We set 
 \begin{equation}\label{gap:1}
  X := \frac{2C_2}{\epsilon} \frac{1}{(GK)^2},
 \end{equation}
 and observe that by Lemma~\ref{lm:trivak} $K \ge (1/2)^{1/4} \ge 1/2$, and so~\eqref{cgap} implies $X \le 1/4$. 

 By Corollary~\ref{cor:bending}
 \begin{multline}\nonumber
  C_2 \cdot 2X \ge \int_X^{2X} B(x) \dx \ge \int_X^{2X} \sum_{k\in\Ip,k \ge GK} a_k^2(x) k^4 \dx
\\ \ge (GK)^2 \int_X^{2X} \sum_{k \in \Ip, k \ge GK} a_k^2(x) k^2 \dx \ge (GK)^2 X \min_{x \in [X,2X]} \sum_{k \in \Ip,k \ge GK} a_k^2(x) k^2,
 \end{multline}
 so that by definition of $X$,
 \begin{equation}\nonumber
  \min_{x \in [X,2X]} \sum_{k \in\Ip, k \ge GK} a_k^2(x) k^2 \le \frac{2C_2}{(GK)^2} = \epsilon X.
 \end{equation}
 Hence, there exists $\bx \in [X/2,X]$ such that $\sum_{k \in \Ip, k \ge GK} a_k^2(2\bx) k^2 \le \epsilon X \le 2 \epsilon \bx$.  It follows from the constraint~\eqref{constr:f} that $\sum_{k\in\Ip} a_k^2(2\bx) k^2 = 4\bx$, and so $a_k \equiv 0$ for $k \in \Ip, K < k < GK$ implies
 \begin{equation}\label{gap:2}
  \sum_{k\in\Ip,k \le K} a_k^2(2\bar x) k^2 \ge 2\bx (2-\epsilon).
 \end{equation}
 Moreover, relation~\eqref{constr:f} implies
 \begin{equation}\label{gap:3}
  \sum_{k\in\Ip,k \le K} a_k^2(\bar x) k^2 \le 2\bx, \qquad \sum_{k\in\Ip,k \le K} a_k^2(3\bar x) k^2 \le 6\bx.
 \end{equation}
 We claim that there exists $k_0 \in \Ip, k_0 \le K$ such that
 \begin{equation}\label{gap:4}
  a_{k_0}(\bx) + a_{k_0}(3\bx) < (2-\delta) a_{k_0}(2\bx).
 \end{equation}
 Indeed, if this were not true, we would have the opposite inequality for all $k \in \Ip, k \le K$, and taking square of such relations, multiplying each by $k^2$, and summing them up would give
 \begin{equation}\nonumber
  \sum_{k\in\Ip,k\le K} a_k^2(\bx)k^2 + a_k^2(3\bx)k^2 + 2a_k(\bx)a_k(3\bx) k^2 \ge (2-\delta)^2 \sum_{k\in\Ip,k\le K} a_k^2(2\bx)k^2.
 \end{equation}
 Using Young's inequality we observe that $2a_k(\bx)a_k(3\bx) k^2 \le \sqrt{3} a_k^2(\bx)k^2 + a_k^2(3\bx)k^2/\sqrt{3}$, and so together with~\eqref{gap:2} and~\eqref{gap:3} we would get
 \begin{equation}\nonumber
  2\bx + 6\bx + 2\sqrt{3} \bx + 6\bx/\sqrt{3} \ge (2-\delta)^2 2\bx(2-\epsilon).
 \end{equation}
 After dividing both sides by $2\bx$, we would get $4 + 2\sqrt{3} \ge (4-4\delta)(2-\epsilon) \ge 8 - 4\epsilon - 8\delta$, which simplifies to $2\epsilon + 4\delta \ge 2 - \sqrt{3}$, a contradiction with the choice of $\epsilon$ and $\delta$. 

 Next we claim that $a_{k_0}'$ is non-increasing in the interval $[0,4\bx]$. 
 To prove it, we observe that for $0 < x \le 4\bx$
 \begin{equation}\nonumber
 x^{-1} \ge (4 X)^{-1} = \frac{1}{4} \frac{\epsilon}{2C_2} (GK)^2 \overset{\eqref{cgap}}{>} \frac{\epsilon}{8C_2} \frac{32C_2}{\epsilon} K^2 = 4 K^2 \overset{K\ge k_0}{\ge} 4k_0^2 > k_0^2,
 \end{equation}
 i.e. $[0,4\bar x] \subset [0,k_0^{-2}]$. Then as in~\eqref{75} we get for any $(\alpha,\beta) \subset [0,4\bx]$ 
 \begin{equation}\nonumber
  a_{k_0}'(\alpha) - a_{k_0}'(\beta) \ge 0.
 \end{equation}

 
 By~\eqref{75} we obtain
 \begin{equation}\label{gap:5}
  a_{k_0}'(3\bx) - a_{k_0}'(\bx) = \int_{\bx}^{3\bx} a_{k_0} k_0^4 \dx - \int_{\bx}^{3\bx} a_{k_0} k_0^2 \ud \lambda \overset{a_{k_0}\ge 0}{\ge} -\int_{\bx}^{3\bx} a_{k_0}(x) k_0^2 \ud \lambda.
 \end{equation}
 Using monotonicity of $a_{k_0}'$ in $[0,4\bx]$ we can estimate the left-hand side of~\eqref{gap:5}
  \begin{equation}\label{gap:6}
   a'_{k_0}(3\bx) - a'_{k_0}(\bx) \le \frac{a_{k_0}(\bx) + a_{k_0}(3\bx) - 2a_{k_0}(2\bx)}{\bx} \overset{\eqref{gap:4}}{<} -\frac{\delta a_{k_0}(2\bx)}{\bx}.
  \end{equation}

  Next, we observe that 
  \begin{equation}\label{gap:7}
    a_{k_0}(x) \le 2a_{k_0}(2\bx), \quad x \in [\bx,3\bx].
  \end{equation}
  Indeed, $a_{k_0}'$ being non-increasing in $(0,4\bx)$ implies that $a_{k_0}$ is concave in $(0,4\bx)$, and so $a_{k_0}(x) > 2a_{k_0}(2\bx)$ for some $x \in [\bx,3\bx]$ would imply either $a_{k_0}(0) < 0$ or $a_{k_0}(4\bx) < 0$ -- a contradiction. 
  We combine~\eqref{gap:5} with~\eqref{gap:6} and use~\eqref{gap:7} to get
  \begin{equation}\label{gap:8}
   \lambda([\bx,3\bx)) \dx \ge \frac{\delta}{2\bx k_0^2}.
  \end{equation}

  To obtain the upper bound on $\lambda([\bx,3\bx))$, we use~\eqref{mu:repre} to show
  \begin{equation}\nonumber
   \mu_{GK}(3\bx) - \mu_{GK}(\bx) = \int_{\bx}^{3\bx} (GK)^2 (1-\mu_{GK}^2(x)) \dx - \lambda([\bx,3\bx)).
  \end{equation}
  Since $\bx \ge X/2 = (C_2/\epsilon) (GK)^{-2} > 2(GK)^{-2}$, by Corollary~\ref{cor:muk:1} the left-hand side of the above relation is bounded by $2$ and $\left|1-\mu_{GK}^2(x)\right| \le 1$ for $x \ge \bx$. Hence
  \begin{equation}\nonumber
   \lambda([\bx,3\bx)) \le 2 + (GK)^2 2\bx \le 3\bx (GK)^2,
  \end{equation}
  where the last inequality follows from $\bx > 2(GK)^{-2}$. 
  The previous estimate together with~\eqref{gap:8} yields
  \begin{equation}\nonumber
   \frac{\delta}{2\bx k_0^2} \le 3\bx (GK)^{2}.
  \end{equation}
  Since $k_0 \le K$ and $\bx \le X = 2C_2\epsilon^{-1} (GK)^{-2}$, a simple algebraic manipulation implies bound on $G$:
  \begin{equation}\nonumber
   G^2 \le \frac{24 C_2^2}{\delta \epsilon^2} \le C_{gap}^2,
  \end{equation}
  a contradiction with~\eqref{cgap}.
\end{proof}
}

In Proposition~\ref{prop:gap} we showed that the ``gap'' $G$ between two non-zero frequencies $a_K \not\equiv 0$, $a_{GK} \not\equiv$ can not be large. For this to be useful it remains to show that the smallest non-zero frequency $k \in \Ip$ for which $a_k \not \equiv 0$ is not very large:

\begin{lemma}\label{lm:nontrivak}
 Let $u \in \A$ be an odd minimizer of $\S$, and $\{a_k\}_{k \in \Ip}$ be the corresponding Fourier coefficients. Then 
 \begin{equation}\nonumber
  \min \{k \in \Ip: a_k \not\equiv 0\} \le \sqrt{\sigma_L} \le 2\sqrt{\sigma_1}.
 \end{equation} 
\end{lemma}
 
\begin{proof} 
 Let $K :=  \min \{k \in \Ip: a_k \not\equiv 0\}$, and the first inequality follows
 \begin{equation}\nonumber
  \sigma_L \ge \int_0^1 \sum_{k\in\Ip} a_k^2(x) k^4 \dx \ge K^2 \int_0^1 \sum_{k \in \Ip} a_k^2(x) k^2 \dx = K^2 \int_0^1 2x \dx = K^2.
 \end{equation}
 The second inequality is~\eqref{sigmale}.
\end{proof}

It follows from~\eqref{lambda:ge} that $\lambda([x_0,2x_0)) \dx \gtrsim 1$. To show the exponential decay of $a_k$ we require a similar upper bound on the Lagrange multiplier $\lambda$:

\begin{lemma}[Upper bound on $\lambda$ on dyadic intervals]\label{lm:lambda:ub}
 There exists a universal constant $C_3$ such that for any $x_0 \in (0,1/2]$ we have
 \begin{equation}\nonumber
  \lambda([x_0,2x_0)) \dx \le C_3.
 \end{equation}
\end{lemma}

\begin{proof}
  By Proposition~\ref{prop:gap} and Lemma~\ref{lm:nontrivak} we can find universal constant $\bar C < \infty$ such that  for every $x_0 \in (0,1/2]$ there exists $k$ such that $a_k \not\equiv 0$ and $1/k^2 \le x_0 \le \bar C/k^2$. Since $x_0 \ge 1/k^2$, by Corollary~\ref{cor:muk:1} we know that $\mu_k(x) \in (-1,1)$ for $x \in [x_0,2x_0]$. Then~\eqref{ELmu} implies
  \begin{equation}\nonumber
   \lambda([x_0,2x_0)) \dx \le \int_{x_0}^{2x_0} k^2 (1-\mu_k^2(x)) \dx + 2 \le k^2 x_0 + 2 \le \bar C + 2.
  \end{equation}
\end{proof}

As an immediate corollary we obtain
\begin{cor}\label{cor:lambda:ub}
 For any $x_0 \in (0,1)$ we have
 \begin{equation}
  \lambda([x_0,1)) \le C_3(|\log_2 x_0|+1).
 \end{equation}
\end{cor}

The following lemma implies the exponential decay of $a_k$:

\begin{lemma}\label{lm:muk:decay}
 There exists a universal constant $C_4$ such that for any $0 < \Delta \le 1/2$ and any frequency $k \in \Ip, a_k \not\equiv 0$  we have 
 \begin{equation}\nonumber
  \left| \left\{ x \in (0,1] : \mu_k(x) \ge -1 + \Delta \right\} \right| \le C_4 \frac{\ln(k)+1}{\Delta k^2}.
 \end{equation}
\end{lemma}

\begin{proof}
 We claim that for $l := \max(4C_{gap}^2,4\sigma_1)$ 
 \begin{equation}\label{claimM}
  (0,1] \subset M := \bigcup_{k \in \Ip, a_k \not\equiv 0} [4k^{-2},lk^{-2}].
 \end{equation}
 Indeed, first we see that 
 \begin{equation}\nonumber
  \max M = \min_{k \in \Ip : a_k \not\equiv 0 } lk^{-2} \overset{Lm~\ref{lm:nontrivak}}{\ge} l/(4\sigma_1) \ge 1. 
 \end{equation}
 Now let us assume that $\alpha := \inf\left\{ x \in [0,1] : (x,\max M] \subset M\right\} > 0$. Then there exists a frequency $k_1 \in \Ip, a_{k_1} \not\equiv 0$ such that $\alpha = 4k_1^{-2}$ and let $k_2 := \min\{ k \in \Ip : k > k_1, a_{k} \not\equiv 0\}$. By Proposition~\ref{prop:gap} we know that $k_2 \le C_{gap}k_1$, and so
 \begin{equation}\nonumber
  \frac{l}{k_2^2} \ge \frac{4 C_{gap}^2}{k_2^2} \ge \frac{4}{k_1^2} = \alpha,
 \end{equation}
 where we used that $l \ge 4C_{gap}^2$. Since $4 k_2^{-2} < \alpha$ and $[4 k_2^{-2},lk_2^{-2}] \cup [\alpha,\max M] = [4k_2^{-2},\max M] \subset M$, we obtain a contradiction. Thus $\alpha = 0$ and~\eqref{claimM} follows. 

 Let $\bar k \in \Ip, a_{\bar k} \not\equiv 0$. By~\eqref{claimM}, for any $x_0 \in [4\bar k^{-2},1]$ we can find $k_0 \in \Ip, a_{k_0} \not \equiv 0, k_0 \le \bar k$ such that $x_0 \in [4k_0^{-2},lk_0^{-2}]$. By Corollary~\ref{cor:muk:bound} there exists $\bar \Delta > 0$ (value of which depends on $l$) such that $\mu_k(x) \le 1 - \bar \Delta$ for any $k \in \Ip, a_k \not\equiv 0$ and $x \in [4k^{-2},lk^{-2}]$. Then, Lemma~\ref{lm:muk:monotone} implies that $\mu_{\bar k}(x_0) \le \mu_{k_0}(x_0) \le 1 - \bar \Delta$. 

 Let us first consider the case $0 < \Delta \le \bar \Delta$. Then we have $\{ x \in [4\bar k^{-2},1] : \mu_{\bar k}(x) \ge -1 + \Delta \} = \{ x \in [4\bar k^{-2},1] : \mu_{\bar k}^2 (x) \le (1 - \Delta)^2 \}$, and so
 \begin{equation}\label{decay:1}
  \bar k^2 \left| \left\{ x \in [4\bar k^{-2},1] : \mu_{\bar k}(x) \ge -1 + \Delta \right\} \right| \left( 1 - (1-\Delta)^2 \right) \le \int_{4\bar k^{-2}}^1 \bar k^2(1-\mu_{\bar k}^2(x)) \dx \le \lambda([4\bar k^{-2},1)) + 2,
 \end{equation}
 where the last inequality follows from \eqref{mu:repre} and Corollary~\ref{cor:muk:1}, which we used to estimate $|\mu_{\bar k}|$ by $1$. Then it follows from Corollary~\ref{cor:lambda:ub} that $\lambda([4\bar k^{-2},1)) \lesssim \ln \bar k + 1$, 
 thus~\eqref{decay:1} implies
 \begin{equation}\label{deltabar}
  \left| \left\{ x \in [0,1] : \mu_{\bar k}(x) \ge -1 + \Delta \right\} \right| \le  4\bar k^{-2} + \left| \left\{ x \in [4\bar k^{-2},1] : \mu_{\bar k}(x) \ge -1 + \Delta \right\} \right| \le \frac{\bar C_4}{\bar k^2 \Delta} \left( \ln \bar k + 1 \right).
 \end{equation}
 
 In the case $1/2 \ge \Delta > \bar \Delta$ we use~\eqref{deltabar} with $\Delta=\bar \Delta$ to show
 \begin{multline}\nonumber
  \left| \left\{ x \in [0,1] : \mu_{\bar k}(x) \ge -1 + \Delta \right\} \right| \le \left| \left\{ x \in [0,1] : \mu_{\bar k}(x) \ge -1 + \bar \Delta \right\} \right| 
  \\ \overset{\eqref{deltabar}}{\le} \frac{\bar C_4}{\bar k^2 \bar \Delta} \left( \ln \bar k + 1 \right) \le \frac{C_4}{\bar k^2 \Delta} \left( \ln \bar k + 1 \right),
 \end{multline}
 where $C_4 = \bar C_4 /(2\bar \Delta)$. 
\end{proof}

\subsection{Regularity estimates -- proof of Theorem~\ref{thm:u}}

We proved above that $\mu_k$ stays close to $-1$ ``most of the time,'' which means that $a_k$ exponentially decay ``most of the time.'' We will now use this fact to obtain estimates for higher derivatives of $u$. The proof is divided into seven steps.

\medskip
{\bf Step 1:} In the following we use decay of $a_k$ to obtain estimates for the $L^2$ norms of the $y$-derivatives of $u$ of any order:
\begin{lemma}\label{lm:reg1} For any $N \in \mathbb{N}$ there exists constant $\bar C_N > 0$ such that for any $x_1 \in (0,1]$:
\begin{equation}\label{cor19.1}
  \lint \left(\frac{\partial^{N+1} u}{\partial y^{N+1}}(x_1,y)\right)^2 \dy = \sum_{k \in \Ip} a_k^2(x_1) k^{2+2N} \le \bar C_N {x_1}^{1-N} (|\ln x_1|^N+1).
\end{equation}
 \end{lemma}

\begin{proof}
 Let $x_1 \in (0,1]$ be fixed, $\alpha > 0$ be large enough so that $f(t) := t^{2N}\exp(-2\alpha t) \le 1$ for $t \ge 1$, and $\kappa \ge 1$ is chosen so that $2(3 C_4 + \alpha) \frac{\ln \kappa + 1}{\kappa^2} + \frac{1}{\kappa^2} = x_1$. 

We split~\eqref{cor19.1} into two pieces
 \begin{align}\label{cor19.2}
  \sum_{k \in \Ip} a_k^2(x_1) k^{2+2N} &= \sum_{k \in \Ip, k > \kappa} a_k^2(x_1) k^{2+2N} + \sum_{k \in \Ip, k \le \kappa} a_k^2(x_1) k^{2+2N}.
 \end{align} We start by estimating the first sum on the right-hand side. Let $x_0 := 1/\kappa^2 < x_1$. Then by Corollary~\ref{cor:muk:1} $\mu_k(x) < 1$ for $x \in (x_0,x_1)$ and any $k \in \Ip, k\ge \kappa$. Therefore, 
 \begin{equation}\nonumber
 \begin{aligned}
  \int_{x_0}^{x_1} \mu_k(x) \dx &= \int\limits_{x_0 < x < x_1 \atop \mu_k(x) \le -1/2} \mu_k(x) \dx + \int\limits_{x_0 < x < x_1 \atop -1/2 < \mu_k(x) \le 1} \mu_k(x) \dx 
\\ &\le -1/2 \left| \left\{ x \in (x_0,x_1) : \mu_k(x) \le -1/2 \right\}\right| + \left| \left\{ x \in (x_0,x_1) : -1/2 < \mu_k(x) \le 1 \right\}\right| 
\\ &= \frac{x_0 - x_1}{2} + 3/2 \left| \left\{ x \in (x_0,x_1) : -1/2 < \mu_k(x) \le 1 \right\}\right| 
\\ &\overset{\textrm{Lm~\ref{lm:muk:decay}}}{\le} \frac{x_0 - x_1}{2} + \frac{3C_4(\ln k + 1)}{k^2} \le \frac{x_0 - x_1}{2} + \frac{3C_4(\ln \kappa + 1)}{\kappa^2} = - \alpha \frac{\ln \kappa + 1}{\kappa^2} \le -\alpha \kappa^{-2},
 \end{aligned}
 \end{equation}
where we used that the function $t \mapsto t^{-2}(\ln t + 1)$ is decreasing for $t \ge 1$. 
Pluging the previous inequality into $\displaystyle a_k^2(x_1) = a_k^2(x_0) \exp \left(2k^2 \int_{x_0}^{x_1} \mu_k(x) \dx\right)$ (see~\eqref{66-}) gives
\begin{equation}\label{cor19.3}
 \sum_{k \in \Ip, k > \kappa} a_k^2(x_1) k^{2+2N} \le \sum_{k \in \Ip,k > \kappa} a_k^2(x_0) k^2 \underbrace{\left( \frac{k}{\kappa} \right)^{2N} \exp\left(-2\alpha k^2 \kappa^{-2} \right)}_{\le 1 \textrm{ since } k/\kappa\ge 1} \kappa^{2N} \overset{\eqref{constr:f}}{\le} 2x_0 \kappa^{2N} = 2\kappa^{2(N-1)}.
\end{equation}
To estimate the second sum on the right-hand side of~\eqref{cor19.2} we just use~\eqref{constr:f} to get
\begin{equation}\label{cor19.4}
 \sum_{k \in \Ip, k\le \kappa} a_k^2(x_1) k^{2+2N} \le \sum_{k \in \Ip} a_k^2(x_1) k^2 \kappa^{2N} = x_1 \kappa^{2N}. 
\end{equation}
To finish, we observe that for $\beta$ large enough (depending on $\alpha$, which depends on $N$) $\kappa_0 := \beta x_1^{-1/2} (|\ln x_1|+1)^{1/2}$ satisfies $2(3C_4 + \alpha)\frac{\ln \kappa_0 + 1}{\kappa_0^2}+\frac{1}{\kappa_0^{2}} < x_1$, which implies $\kappa \le \kappa_0$. By plugging~\eqref{cor19.3} and~\eqref{cor19.4} into~\eqref{cor19.2} and using $\kappa \le \beta x_1^{-1/2} (|\ln x_1|+1)^{1/2}$ we obtain
\begin{multline}\nonumber
 \sum_{k\in\Ip} a_k^2(x_1) k^{2+2N} \le x_1 \kappa^{2N} + 2\kappa^{2(N-1)} \lesssim x_1 \left( \beta^2 x_1^{-1} (|\ln x_1|+1) \right) ^{N} + 2 \left( \beta^2 x_1^{-1} (|\ln x_1|+1) \right) ^{N-1} 
\\ \lesssim x_1^{1-N} (|\ln x_1|+1)^{N}.
\end{multline}
\end{proof}

{\bf Step 2:} Next we show that $\lambda$ restricted to the interval $(0,1)$ is a locally integrable function, i.e. after a slight abuse of notation we have:
\begin{lemma}\label{lm:lambdafunction}
 There exists a function $\lambda(x) \in L^1_{loc}((0,1])$ such that for any $[\alpha,\beta) \subset (0,1)$
 \begin{equation}\nonumber
  \lambda([\alpha,\beta)) = \int_\alpha^\beta \lambda(x) \dx. 
 \end{equation}
\end{lemma}

\begin{proof}
 First we observe that by Lemma~\ref{lm:trivak} $a_k \equiv 0$ for $k \le (1/2)^{1/4}$, and so
 \begin{equation}\nonumber
  \sum_{k \in \Ip} a_k^2(x) k^4 \ge (1/2)^{1/2} \sum_{k \in \Ip} a_k^2(x) k^2 \overset{\eqref{constr:f}}{=} \sqrt{2} x.
 \end{equation}
 Using~\eqref{66} and the previous relation we get
 \begin{equation}\label{103}
  \int_0^1 \sum_{k \in \Ip} (a_k'^2 k^2 + a_k^2 k^6) \Psi \dx = \int_0^1 \sum_{k \in \Ip} a_k^2k^4 \Psi \ud \lambda \gtrsim \int_0^1 x\Psi \ud \lambda
 \end{equation}
 for any non-negative test function $\Psi \in \mathcal{D}(0,1)$. Since by Lemma~\ref{cor:muk:1} $|\mu_k(x)|\le 1$ for $x \ge k^{-2}$, we have $|a_k'(x)|^2 k^2 \le a_k^2(x) k^6$ if $x \ge k^{-2}$. Thus 
 \begin{equation}\nonumber
  \sum_{k \in \Ip} a_k'^2(x)k^2 + a_k^2(x)k^6 \le x^{-1} \sum_{k \in \Ip} a_k'^2(x) + 2 \sum_{k \in \Ip} a_k^2(x)k^6 \overset{Lm~\ref{lm:reg1}}{\le} x^{-1} \sum_{k \in \Ip} a_k'^2(x) + C(|\ln x|+1).
 \end{equation}
 Hence, for any non-negative test function $\varphi \in \mathcal{D}(0,1)$, the choice of $\Psi(x) := \varphi(x)/x$ in~\eqref{103} and the previous relation imply
 \begin{equation}\nonumber
  \int_0^1 \varphi \ud \lambda \lesssim \int_0^1 \left[ \frac{m(x)}{x^2} + \frac{|\ln x|+1}{x} \right] \varphi(x) \dx,
 \end{equation}
 where $m = \sum_{k \in \Ip} a_k'^2k^2 \in L^1(0,1)$. This concludes the proof of the lemma. 
\end{proof}

{\bf Step 3:} Since $\lambda$ is a (locally integrable) function, we use \eqref{EL} to show that $a_k''$ is (locally) $L^1$, which in turn allows us to test~\eqref{EL} with $a_k'$ to obtain the following:

\begin{lemma}
We have
\begin{equation}\label{cor20.1-}
  \lint u_{,x}^2(x_1,y) \dy = \sum_{k\in\Ip} a_k'^2(x_1) \lesssim |\ln x_1|+1.
\end{equation}
\end{lemma}

\begin{proof}
Since the Lagrange multiplier $\lambda$ is a locally integrable function, the Euler-Lagrange equation~\eqref{EL} implies
\begin{equation}\label{EL2}
 -a_k''(x) = a_k(x)k^4 - a_k(x)k^2 \lambda (x) \in L^1_{loc}((0,1]).
\end{equation}
As a consequence, we obtain that $a_k' \in L^{\infty}_{loc}((0,1])$, which then justifies the following computation:
\begin{equation}\label{cor20.3}
 \frac{1}{2} \Bigg( \sum_{k\in\Ip} a_k'^2(x) - a_k^2(x) k^4 \Bigg)' =  \sum_{k\in\Ip}  a''_k(x) a'_k(x) - a'_k(x) a_k(x) k^4 = \lambda(x) \sum_{k\in\Ip} a'_k(x) a_k(x) k^2 = \lambda(x),
 \end{equation}
where the last inequality is obtained by differentiating~\eqref{constr:f}. Since $\displaystyle\int_0^1 \Bigg( \sum_{k\in\Ip} a_k'^2 + a_k^2k^4\Bigg) \dx = \sigma_L \lesssim 1$, we can find $x_0 \in (1/2,1)$ such that $\displaystyle\sum_{k\in\Ip} a_k^2(x_0) k^4 - a_k'^2(x_0) \lesssim 1$. Integrating~\eqref{cor20.3} from $x_0$ to $x_1$ then implies
\begin{multline}\label{cor20.4} 
 \Bigg|\sum_{k\in\Ip} a_k^2(x_1) k^4 - a_k'^2(x_1)\Bigg| \le \Bigg| \int_{x_0}^{x_1}  \Bigg( \sum_{k\in\Ip} a_k'^2(x) - a_k^2(x) k^4 \Bigg)' \dx \Bigg| + \Bigg|\sum_{k\in\Ip} a_k^2(x_0) k^4 - a_k'^2(x_0)\Bigg|
\\ \le 2\Bigg|\int_{x_0}^{x_1} \lambda(x) \dx\Bigg|+ C \le 2\Bigg|\int_{x_1}^{1} \lambda(x) \dx\Bigg| + C'.
\end{multline}
By Corollary~\ref{cor:lambda:ub} 
$\displaystyle\left|\int_{x_1}^1 \lambda(x) \dx\right| \lesssim |\ln x_1| + 1$, thus~\eqref{cor20.4} together with Lemma~\ref{lm:reg1} imply~\eqref{cor20.1-}.
\end{proof}

 {\bf Step 4:} For any $x_1 \in (0,1]$
\begin{align}\label{cor20.1}
  \lint u_{,xy}^2(x_1,y) \dy = \sum_{k \in \Ip} a_k'^2(x_1) k^2 &\lesssim x_1^{-1} (|\ln x_1|^2+1),\\
  \lint u_{,xyy}^2(x_1,y) \dy = \sum_{k \in \Ip} a_k'^2(x_1) k^4 &\lesssim x_1^{-2} (|\ln x_1|^3+1).\label{cor20.1+}
\end{align}
We start with the proof of~\eqref{cor20.1}. For $\kappa := x_1^{-1/2}$ we write
 \begin{equation}\label{cor20.2}
  \sum_{k \in \Ip} a_k'^2(x_1) k^2 = \sum_{k \in \Ip, k > \kappa} a_k'^2(x_1) k^2 + \sum_{k \in \Ip, k \le \kappa} a_k'^2(x_1) k^2.
\end{equation}
Since by Corollary~\ref{cor:muk:1} $\mu^2_k(x) \le 1$ for $x \ge 1/k^2$ and $a'_k(x) = \mu_k(x) a_k(x) k^2$, the first sum on the right-hand side can be estimated by
\begin{equation}\nonumber
 \sum_{k \in \Ip, k > \kappa} a'^2_k(x_1) k^2 \le \sum_{k\in\Ip,k > \kappa} a_k^2(x_1) k^6 \underset{\textrm{with } N=2}{\overset{\textrm{Lm \ref{lm:reg1}}}{\lesssim}}  x_1^{-1} (|\ln x_1|^2+1).
\end{equation}
From~\eqref{cor20.1-} we know 
$\sum_{k \in\Ip, k \le \kappa} a'^2_k(x_1) k^2 \le \kappa^2 \sum_{k\in\Ip} a'^2_k(x_1) = x_1^{-1} \sum_{k\in\Ip} a'^2_k(x_1) \lesssim x_1^{-1} (|\ln x_1|+1)$. Combining these two estimates with~\eqref{cor20.2} gives~\eqref{cor20.1}.  

The proof of~\eqref{cor20.1+} is based on the same ideas as the proof of~\eqref{cor20.1} -- 
we split~\eqref{cor20.1+} into two sums to get:
 \begin{multline}\nonumber
  \sum_{k\in\Ip} a_k'^2(x_1) k^4 = \sum_{k\in\Ip,k > \kappa} a_k'^2(x_1) k^4 + \sum_{k\in\Ip, k \le \kappa} a_k'^2(x_1) k^4 \le \sum_{k\in\Ip,k > \kappa} a_k^2(x_1) k^8 + \kappa^4 \sum_{k\in\Ip} a_k'^2(x_1) 
\\ \lesssim x_1^{-2}(|\ln x_1|^3 + 1) + x_1^{-2} (|\ln x_1| + 1).
\end{multline}

{\bf Step 5:} We claim that
 \begin{equation}\label{lambda:ub}
  \lambda(x) \lesssim x^{-1} (|\ln x|^3 + 1) \quad \textrm{for a.e. } x \in (0,1).
 \end{equation}
 First we recall~\eqref{66}, which states that for any $\Psi \in \mathcal{D}((0,1))$, $\Psi\ge0$:
 \begin{equation}\label{lambda:der}
  \int_0^1 \lambda(x) \left( \sum_{k\in\Ip} a_k^2(x) k^4\right) \Psi(x) = \int_0^1 \left( \sum_{k\in\Ip} a_k^2(x) k^6 + a_k'^2(x) k^2 \right) \Psi(x) \dx.
 \end{equation}
By Lemma~\ref{lm:reg1} and~\eqref{cor20.1}, the argument in the parenthesis on the right-hand side of the previous relation is bounded by a multiple of $x^{-1} (|\ln x|^2+1)$, and~\eqref{lambda:ub} follows from the following lower bound on $\sum_{k\in\Ip} a_k^2(x) k^4$:
 \begin{align*}\nonumber
  (2x)^2 &= \left( \sum_{k\in\Ip} a_k^2(x) k^2 \right)^2 \overset{Cauchy- \atop Schwarz}{\le} \left( \sum_{k\in\Ip} a_k^2(x) \right) \left( \sum_{k\in\Ip} a_k^2(x) k^4 \right) 
\\ & \!\!\!\!\!\!\!\!\overset{a_k(0)=0}{=} \left( \sum_{k\in\Ip} \left[ \int_0^x a_k'(x') \ud x' \right] ^2\right) \left( \sum_{k\in\Ip} a_k^2(x) k^4 \right) \overset{\text{H\"older's}}{\le} x \left( \int_0^x  \sum_{k\in\Ip} a_k'^2(x') \ud x' \right) \cdot \left( \sum_{k\in\Ip} a_k^2(x) k^4 \right) 
\\ &\!\!\overset{\eqref{cor20.1-}}{\lesssim} x \left( \int_0^x |\ln x'|+1 \ud x' \right) \left( \sum_{k\in\Ip} a_k^2(x) k^4 \right) 
\\ &= x^2 (|\ln x|+2) \left( \sum_{k\in\Ip} a_k^2(x) k^4 \right).
 \end{align*}

{\bf Step 6:} For (a.e.) $x \in (0,1)$
 \begin{equation}\label{adotdot}
  \lint u_{,xx}^2(x,y) \dy =\sum_{k\in\Ip} a_k''^2(x) \lesssim x^{-2} (|\ln x|^7+1).
 \end{equation}
  Since $\lambda(x) \ge 0$ a.e. in $(0,1)$, the Euler-Lagrange equation $a_k''(x) = a_k(x) k^4 - \lambda(x) a_k(x) k^2$ 
implies that a.e. in $(0,1)$:
 \begin{equation}\nonumber
  \sum_{k\in\Ip} a_k''^2(x) \le \sum_{k\in\Ip} a_k^2(x) k^8 + \lambda(x)^2 \sum_{k\in\Ip} a_k^2(x) k^4.
 \end{equation}
 Then~\eqref{adotdot} follows from Lemma~\ref{lm:reg1} (applied with $N=1$ and $N=2$) and \eqref{lambda:ub}.

 We showed that~\eqref{adotdot} holds for a.e. $x \in (0,1)$, which is sufficient for the proof of the upper bound (see Section~\ref{sect:ub}). For completeness let us sketch the idea how to extend it to all $x \in (0,1)$. To obtain that it is enough to show that the Lagrange multiplier $\lambda$ is a continuous function on $(0,1)$. For that we can use the results from this step (i.e.~\eqref{adotdot} for a.e. $x \in (0,1)$) and the previous steps to show that $\sum_{k\in\Ip} a_k^2(x) k^6 + a_k'^2(x) k^2 / \sum_{k\in\Ip} a_k^2(x) k^4$ has a derivative in $L^\infty_{loc}((0,1])$, which implies that also $\lambda' \in L^\infty_{loc}((0,1])$, in particular it is continuous in $(0,1]$.  

{\bf Step 7:} We claim that
 \begin{equation}\label{a}
  \lint u^2(x,y) \dy =\sum_{k\in\Ip} a_k^2(x) \lesssim x^2 (|\ln x|+1).
 \end{equation}
 We use $a_k(0)=0$ to write
 \begin{multline}\nonumber
  \sum_{k\in\Ip} a_k^2(x) = \sum_{k\in\Ip} \left( \int_0^x a_k'(x') \ud x'\right)^2 \overset{\text{H\"older's}}{\le} x \int_0^x \Bigg( \sum_{k\in\Ip} a_k'^2(x') \Bigg) \ud x'
\\ \overset{\eqref{cor20.1-}}{\lesssim} x \int_0^x (|\ln x'|+1) \ud x' = x^2 (|\ln x|+2).
 \end{multline}
 
\section{Proof of Lemma~\ref{lm:deltal}}\label{pf:lm41}

\renewcommand{\ve}{\varphi_\eta} 

In this section we prove Lemma~\ref{lm:deltal}, which we used in the proof of the lower bound (see Section~\ref{sect:lb}):

\noindent\usebox{\lemboxA}

\begin{proof}
Lemma~\ref{lm:relaxed:constrain} says that
\begin{equation}
 \min_{\sum_{k\in\Ip} a_k^2(x)k^2 = 2x} \S(u) =  \min_{\stackrel{\sum_{k\in\I} a_k^2(x)k^2 \ge 2x}{a_k(0)=0}} \S(u).
\end{equation} 
Hence, to prove~\eqref{deltal}, given $v : [-1,1]\times[-L,L] \to \R$, a $2L$-periodic function in $y$, it is enough to construct a function $u : [0,1] \times [-L,L] \to \R$, which is $2L$-periodic in $y$, and satisfies 
\begin{align}\nonumber
\lint u_{,y}^2(x,y) \dy &\ge 2x, \quad x \in (0,1],
\\ \lint u_{,y}^2(0,y) \dy &= 0, \label{bkzero}
\end{align}
together with
\begin{equation}\label{60}
 \S(u) \le  \S(v) + L^2 \int_{-1}^1 \left( \lint v_{,y}^2(x,y)/2 \dy  - \vt(x) \right)^2 \dx + \delta_L,
\end{equation}
where $\vt(x) = x \chi_{[0,1]}(x)$ was defined in~\eqref{upsilon}.
If $\S(v) + L^2 \int_{-1}^1 \left( \lint v_{,y}^2(x,y)/2 \dy - \vt(x) \right)^2 \dx \ge 4\sigma_1$, \eqref{60} immediately follows (observe that by Lemma~\ref{lm:sigmal} $\sigma_L \le 4\sigma_1$). Therefore we can assume that
\begin{equation}\label{le:sigmal}
 \int_{-1}^1 \left( \lint v_{,y}^2(x,y)/2 \dy  - \vt(x) \right)^2 \dx \le 4\sigma_1 L^{-2}.
\end{equation}
Let $0 < \eta < \pi^{-6}$ and
\begin{equation}\nonumber
 \ve(x) := \frac{1}{\eta} \varphi\left( \frac{x - 2\eta}{\eta} \right),
\end{equation}
where $\varphi \in C^{\infty}(\R)$ is a standard smoothing kernel with compact support in $(-1,1)$, which is even, $0 \le \varphi \le 1$, and $\displaystyle \int_{-1}^1 \varphi = 1$.
Since $v$ is $2L$-periodic in $y$, we can use Fourier series to write 
\begin{equation}\nonumber
 v(x,y) = \sum_{k \in \I, k>0} a_k(x)\sin(\pi ky) + \sum_{k \in \I, k\le 0} a_k(x)\cos(\pi ky).
\end{equation}
For $k \in \Ip$ and $x \in [0,1]$ we define
\begin{equation}\nonumber
 b_k(x) := \sqrt{ \left( a_k^2 + a_{-k}^2 \right) * \ve (x)},
\end{equation}
where $*$ stands for the convolution. 
We observe that since $\eta < 1/3$ and $a_k$ is defined in $[-1,1]$, $b_k$ is well-defined. Moreover, by following the proof of Lemma~\ref{lm:sublinear} we observe that 
\begin{equation}\label{bk:est}
 \int_0^1 \sum_{k\in\Ip} b_k'^2(x) + b_k^2 k^4 \dx \le \int_{-1}^1 \sum_{k\in\I} a_k'^2(x) + a_k^2 k^4 \dx = \S(v).
\end{equation}
For the construction to satisfy~\eqref{bkzero} we would need $\sum_{k\in\Ip} b_k^2(0) k^2 = 0$, but it can be only shown that $\sum_{k\in\Ip} b_k^2(0) k^2$ is small (see~\eqref{187} below). Because of that we will change $b_k$ in few steps. 
By repeating the proof of Lemma~\ref{lm:trivak} we obtain that for any $k \in \Ip$
\begin{equation}\label{187}
\int_0^1 b_k'^2 \dx \le 4 \int_0^1 b_{k}^2 k^4 \dx,
\end{equation}
or otherwise we can replace $b_k$ and decrease the energy. 
Since $\supp \ve \subset (\eta,3\eta)$, we get for $x \in [0,\eta]$:
\begin{multline}
 \sum_{k\in\Ip} b_k^2(x) k^2 = \sum_{k \in \I} (a_k^2 * \ve)(x)k^2 = \int_{-1}^0 \sum_{k \in \I} a_k^2(x')k^2 \ve(x-x') \ud x'
\\ \overset{\textrm{H\"older's}}{\le} \Bigg( \int_{-1}^{0} \Bigg( \sum_{k \in \I} a_k^2(x) k^2 \Bigg)^2 \dx \Bigg)^{1/2} \Bigg( \int_{-1}^1 \ve^2(x) \dx \Bigg)^{1/2} \lesssim L^{-1} \eta^{-1/2},\label{188}
\end{multline}
where the last inequality follows from $\displaystyle \lint v_{,y}^2(x,y) \dy = \sum_{k \in \I} a_k^2(x) k^2$ and \eqref{le:sigmal}.
For $k \in \Ip, k \le 1/2$, $b_k \not\equiv 0$, we have
\begin{equation}\nonumber
 \max_{x \in [0,1]} b_k - \min_{x \in [0,1]} b_k \le \int_0^1 |b'_k| \dx \overset{\textrm{H\"older's}}{\le} \left( \int_0^1 b_k'^2 \dx \right)^{1/2} \overset{\eqref{187}}{\le} 2k^2 \left( \int_0^1 b_k^2 \dx \right)^{1/2} \overset{k \le 1/2}{\le} \max_{x \in [0,1]} b_k/2,
\end{equation}
which implies $\max_{x \in [0,1]} b_k(x) \le 2 \min_{x \in [0,1]} b_k(x)$. Therefore, for any $x \in [0,1]$:
\begin{equation}\label{190}
 \sum_{k \in \Ip, k \le 1/2} b_k^2(x) k^2 \le 2 \sum_{k \in \Ip, k \le 1/2} b_k^2(0) k^2 \overset{\eqref{188}}{\lesssim} L^{-1}\eta^{-1/2}.
\end{equation}
For $k \in \Ip, k \le 1/2$ we define $c_k(x) := 0$, for $k \in \Ip, 1/2 < k < \eta^{-1/2}$  
\begin{equation}\nonumber
 c_k(x) := \left\{ \begin{array}{ll}
 \frac{x}{\eta} b_k(\eta) \ & x \in [0,\eta]
\\			 b_k(x) & x \in [\eta,1],
\end{array} \right.
\end{equation}
and for $k \in \Ip, k \ge \eta^{-1/2}$ 
\begin{equation}\nonumber
 c_k(x) := \left\{ \begin{array}{ll}
 0 \ & x \in [0,\eta-k^{-2}]
\\ k^2(x-\eta+k^{-2}) b_k(\eta) & x \in [\eta-k^{-2},\eta]
\\ b_k(x) & x\in [\eta,1].           
\end{array} \right.
\end{equation}
To estimate the energy of $c_k$ we start with $k \in \Ip, k\in (1/2,\eta^{-1/2})$:
\begin{equation}\nonumber
 \int_0^1 c_k'^2 + c_k^2 k^4 \dx = \left( \eta^{-1} + k^4 \eta/3 \right) b_k^2(\eta) + \int_{\eta}^1 b_k'^2 + b_k^2 k^4 \dx \le 5\eta^{-1}k^2 b_k^2(\eta) + \int_{0}^1 b_k'^2 + b_k^2 k^4 \dx,
\end{equation}
where in the last inequality we used that $k \in (1/2,\eta^{-1/2})$, and so $\eta^{-1} \le 4\eta^{-1}k^2$ and $k^4\eta/3 \le \eta^{-1}k^2$. If $k \in \I, k \ge \eta^{-1/2}$, we get
\begin{equation}\nonumber
 \int_0^1 c_k'^2 + c_k^2 k^4 \dx = \frac{4}{3} b_k^2(\eta)k^2 + \int_{\eta}^1 b_k'^2 + b_k^2 k^4 \dx \le \frac{4}{3} b_k^2(\eta) k^2 + \int_{0}^1 b_k'^2 + b_k^2 k^4 \dx.
\end{equation}
We use two previous inequalities to obtain
\begin{multline}\label{195}
 \int_0^1 \sum_{k \in \Ip} c_k'^2 + c_k^2 k^4 \dx \le \left( 5\eta^{-1} + \frac{4}{3} \right) \sum_{k\in \Ip} b_k^2(\eta) k^2 + \int_0^1 \sum_{k\in \Ip} b_k'^2 + b_k^2 k^4 \dx 
\\ \overset{\eqref{188}}{\le} CL^{-1}\eta^{-3/2} + \int_0^1 \sum_{k\in \Ip} b_k'^2 + b_k^2 k^4 \dx.
\end{multline}
Next we want to compare $\sum_{k \in \Ip} c_k^2(x) k^2$ with $2x$. For $x \in [0,1]$, we define the auxiliary function 
\begin{equation}\nonumber
 \psi(x) := \int_0^1 2x' \ve(x-x') \ud x' = 2(\vt * \ve) (x),
\end{equation}
and set
\begin{equation}\label{r1}
 r(x) := 2x - \sum_{k \in \Ip} c_k^2(x)k^2.
\end{equation}
It follows from the definition of $c_k$ that for $x \in [\eta,1]$:
\begin{multline}\label{bk:constr} 
 r(x) = 2x - \sum_{k \in \Ip, k>1/2} b_k^2(x)k^2 = 2x - \sum_{k \in \Ip} b_k^2(x)k^2 + \sum_{k \in \Ip, k \le 1/2} b_k^2(x)k^2  
\\ = (2x - \psi(x)) - \Bigg(\Bigg( \sum_{k \in \I} a_k(\cdot)^2 k^2 - 2\vt(\cdot) \Bigg) * \ve\Bigg)(x) + \sum_{k \in \Ip, k \le 1/2} b_k^2(x) k^2. 
\end{multline}
Then~\eqref{le:sigmal} implies
\begin{equation}\nonumber
 \frac{1}{4} \int_{-1}^{1} \Bigg( \sum_{k\in\I} a_k^2(x)k^2 - 2\vt(x) \Bigg)^2 \dx = \int_{-1}^1 \left( \lint v_{,y}^2(x,y)/2 \dy  - \vt(x) \right)^2 \dx \le 4\sigma_1 L^{-2},
\end{equation}
and so 
\begin{equation}\nonumber
 \Bigg| \Bigg(\Bigg( \sum_{k \in \I} a_k(\cdot)^2 k^2 - 2\vt(\cdot) \Bigg) * \ve\Bigg)(x) \Bigg| \le \Bigg\| \sum_{k\in\I} a_k^2(\cdot)k^2 - 2\vt(\cdot) \Bigg\|_{L^2(-1,1)} \left\| \ve \right\|_{L^2} \lesssim L^{-1} \left\| \ve \right\|_{L^2}.
\end{equation}
Hence, it follows from \eqref{bk:constr} that for every $x\in [\eta,1]$:
\begin{gather}
|r(x) - (2x-\psi(x))| \overset{\eqref{190}}{\lesssim} L^{-1} \left\| \ve \right\|_{L^2} + L^{-1}\eta^{-1/2} \lesssim L^{-1} \eta^{-1/2}.\label{200}
\end{gather}
It follows from the definition of $r$ (see~\eqref{r1}) that $r(x) \le 2x$ for all $x \in [0,1]$. We also observe that $\psi(x) = 2x - 4\eta$ for $x \in [3\eta,1]$, and so $2x-\psi(x) = 4\eta$ for $x \in [3\eta,1]$. Since $\eta < \pi^{-6} \le (1/3)^{3}$ implies $\eta^{2/3} \ge 3\eta$, it follows from \eqref{200} that $r(x) \le 4\eta + \bar CL^{-1}\eta^{-1/2}$ for $x \in [\eta^{2/3},1]$. To summarize, we have
\begin{equation}\label{202}
 r(x) \le \left\{ \begin{array}{ll} 2x & \quad x \in [0,\eta^{2/3}] 
\\ 4\eta + \bar CL^{-1}\eta^{-1/2} & \quad x \in [\eta^{2/3},1]. \end{array} \right.
\end{equation}

Now we would like to define $d_k(x)$ such that $\sum_{k \in \Ip} d_k^2(x) k^2 \ge r(x)$,  $\sum_{k \in \Ip} d_k^2(0) k^2 = 0$, and such that its energy $\displaystyle \int_0^1 \sum_{k\in\Ip}  d_k'^2(x) + d_k(x)^2 k^4 \dx$ is small. To do that we first use Lemma~\ref{lm:construction} with $b = \eta^{2/3}$ to obtain $u$ (with coefficients $e_k$) such that  $\sum_{k \in \Ip} e_k^2(x) k^2 = 2x$ for $x \in [0,\eta^{2/3}]$ and $\int_0^1 \sum_{k\in\Ip} e_k'^2 + e_k^2 k^4 \lesssim \eta^{2/3}$.

Since $\eta < \pi^{-6}$ and $L\ge1$, we have that $\eta^{-1/6} > \pi/L$. Therefore we can find $k_0 \in \Ip, \eta^{-1/6} \le k_0 \le 2\eta^{-1/6}$ and define
\begin{equation}\nonumber
 f_{k_0}(x) := \begin{cases} \frac{x}{\eta^{2/3}} \frac{\sqrt{4\eta + \bar C L^{-1} \eta^{-1/2}} }{k_0} & x \in [0,\eta^{2/3}]
\\ \frac{\sqrt{4\eta + \bar C L^{-1} \eta^{-1/2}} }{k_0} & x \in [\eta^{2/3},1].
               \end{cases}
\end{equation}
We set $f_k \equiv 0$ for all $k \in \Ip, k \neq k_0$. Then
\begin{equation}\nonumber
 \int_0^1 f_{k_0}'^2 + f_{k_0}^2 k_0^4 \lesssim ( 4\eta + \bar C L^{-1} \eta^{-1/2} ) \eta^{-1/3}
\end{equation}
and for $x \in [\eta^{2/3},1]$
\begin{equation}\nonumber
 f_{k_0}^2(x) k_0^2 = 4\eta + \bar C L^{-1} \eta^{-1/2} \ge r(x).
\end{equation}
Therefore, by using Lemma~\ref{lm:sublinear} to combine $e_k$ and $f_k$ we obtain $d_k(x) = \sqrt{e_k^2(x) + f_k^2(x)}$ such that  $\sum_{k \in \Ip} d_k^2(x) k^2 \ge r(x)$, $\sum_{k \in \Ip} d_k^2(0) k^2 = 0$, and $\displaystyle \int_0^1 \sum_{k\in\Ip} d_k'^2(x) + d_k(x)^2 k^4 \dx \lesssim \eta^{-1/3} \left( 4\eta + \bar C L^{-1} \eta^{-1/2} \right)$. 
Finally, we use Lemma~\ref{lm:sublinear} again to combine $c_k$ and $d_k$ into $g_k = \sqrt{c_k^2 + d_k^2}$ so that 
\begin{equation}\nonumber
 \sum_{k \in \Ip} g_k^2(x) k^2 = \sum_{k \in \Ip} c_k^2(x) k^2 + d_k^2(x) k^2 \ge \sum_{k \in \Ip} c_k^2(x) k^2 + r(x) \overset{\eqref{r1}}{=} 2x
\end{equation}
and
\begin{multline}\nonumber
 \int_0^1 \sum_{k\in\Ip} g_k'^2(x) + g_k(x)^2 k^4 \dx \overset{\eqref{195}}{\le} C\left( L^{-1}\eta^{-3/2} + \eta^{2/3} + L^{-1} \eta^{-5/6} \right) + \int_0^1 \sum_{k\in\Ip}  b_k'^2(x) + b_k(x)^2 k^4 \dx 
\\ \overset{\eqref{bk:est}}{\le} C\left( L^{-1}\eta^{-3/2} + \eta^{2/3} + L^{-1} \eta^{-5/6} \right) + \S(v).
\end{multline}
To conclude we observe that $\sum_{k \in \Ip} g_k^2(0) k^2 = 0$ and that by choosing $\eta := \min(L^{-1/2},\pi^{-6}/2)$ we get that $L^{-1}\eta^{-3/2} + \eta^{2/3} + L^{-1} \eta^{-5/6} \to 0$ as $L \to \infty$.
\end{proof}

\section*{Acknowledgments}
The author would like to thank M. Goldman and F. Otto for interesting discussions. 

\bibliographystyle{amsplain}
\bibliography{bella}

\end{document}